\documentclass[journal,onecolumn]{IEEEtran}
\usepackage{graphicx}
\usepackage{dcolumn}
\usepackage{bm}
\usepackage{comment}
\usepackage{fancyhdr}
\usepackage{latexsym}
\usepackage{graphicx,color}
\usepackage{epstopdf}  
\usepackage[pdfstartview=FitH]{hyperref}
\usepackage{xargs}
\usepackage{shadow,epsf,amsthm,amssymb,amsmath,caption,mathtools}
\usepackage{cleveref}

\usepackage{threeparttable}
\usepackage{multirow}
\usepackage{float}
\usepackage{makecell}
\usepackage{booktabs, tabularx, colortbl}
\usepackage{array}
\usepackage{subfigure}
\newcolumntype{C}{>{$}c<{$}}
\AtBeginDocument{
	\heavyrulewidth=.08em
	\lightrulewidth=.05em
	\cmidrulewidth=.03em
	\belowrulesep=.65ex
	\belowbottomsep=0pt
	\aboverulesep=.4ex
	\abovetopsep=0pt
	\cmidrulesep=\doublerulesep
	\cmidrulekern=.5em
	\defaultaddspace=.5em
}

\usepackage{enumerate}

\usepackage{setspace}                           

\usepackage{lipsum}

\usepackage{tikz}	
\usetikzlibrary{backgrounds,fit,decorations.pathreplacing}  

\usepackage[colorinlistoftodos,prependcaption]{todonotes}

\usepackage[ruled,vlined,noresetcount]{algorithm2e}
\SetKwInput{KwOutput}{Output}              
\SetKwInOut{Input}{Input}
\SetKwInput{Output}{Output} 

\theoremstyle{plain}

\usepackage{lipsum}
\usepackage{amsfonts}
\usepackage{graphicx}
\usepackage{epstopdf}
\usepackage{algorithmic}
\usepackage{bm}
\usepackage{enumerate}
\usepackage{mathrsfs}
\ifpdf
  \DeclareGraphicsExtensions{.eps,.pdf,.png,.jpg}
\else
  \DeclareGraphicsExtensions{.eps}
\fi

\newtheorem{definitionenv}{Definition}
\newtheorem{lemmaenv}[definitionenv]{Lemma}
\newtheorem{theoremenv}[definitionenv]{Theorem}
\newtheorem{corollaryenv}[definitionenv]{Corollary}
\newtheorem{propositionenv}[definitionenv]{Proposition}
\newtheorem{conjectureenv}[definitionenv]{Conjecture}
\newtheorem{remarkenv}[definitionenv]{Remark}
\newenvironment{remark}{\begin{remarkenv}\rm}{\end{remarkenv}}
\newcommand{\br}{\begin{remark}}
	\newcommand{\er}{\end{remark}}

\newtheorem{exampleenv}{Example}
\newtheorem{app-lemmaenv}[section]{Lemma}

\newenvironment{definition}{\begin{definitionenv}\rm}{\end{definitionenv}}
\newenvironment{lemma}{\begin{lemmaenv}\rm}{\end{lemmaenv}}
\newenvironment{theorem}{\begin{theoremenv}\rm}{\end{theoremenv}}
\newenvironment{corollary}{\begin{corollaryenv}\rm}{\end{corollaryenv}}
\newenvironment{example}{\begin{exampleenv}\rm}{\end{exampleenv}}
\newenvironment{proposition}{\begin{propositionenv}\rm}{\end{propositionenv}}
\newenvironment{conjecture}{\begin{conjectureenv}\rm}{\end{conjectureenv}}
\newenvironment{app-lemma}{\begin{app-lemmaenv}\rm}{\end{app-lemmaenv}}

\newcommand{\bd}{\begin{definition}}
	\newcommand{\ed}{\end{definition}}
\newcommand{\bl}{\begin{lemma}}
	\newcommand{\el}{\end{lemma}}
\newcommand{\elp}{\hspace*{\fill} $\Box$
\end{lemma}}
\newcommand{\bt}{\begin{theorem}}
\newcommand{\et}{\end{theorem}}
\newcommand{\etp}{\hspace*{\fill} $\Box$
\end{theorem}}
\newcommand{\bc}{\begin{corollary}}
\newcommand{\ec}{\end{corollary}}
\newcommand{\ecp}{\hspace*{\fill} $\Box$
\end{corollary}}
\newcommand{\bcj}{\begin{conjecture}}
\newcommand{\ecj}{\end{conjecture}}

\newcommand{\be}{\begin{example}}
\newcommand{\ee}{\end{example}}
\newcommand{\eep}{\hspace*{\fill} $\Box$
\end{example}}
\newcommand{\bp}{\begin{proposition}}
\newcommand{\ep}{\end{proposition}}
\newcommand{\epp}{
\end{proposition}}

\newcommand{\eeq}{ \setcounter{equation} {\value{enumi}}}

\def\beq{\begin{equation}}
\def\eeq{\end{equation}}

\def\bean{\begin{IEEEeqnarray*}{rCl}}
\def\eean{\end{IEEEeqnarray*}}

\ifCLASSINFOpdf
\else
\fi
\begin{document}
%
\title{RIP-based Performance Guarantee for Low Rank Matrix Recovery via $L_{*-F}$ Minimization}
%
%
%

\author{Yan Li
           and Liping Zhang
           \thanks{This work was supported in part by the National Natural Science
Foundation of China under Grant No. 12171271 (Corresponding author: Liping Zhang)}
\thanks{Yan Li and Liping Zhang are with the Department of Mathematical Sciences, Tsinghua University, Beijing 100084, China (email:li-yan20@mails.tsinghua.edu.cn; lipingzhang@tsinghua.edu.cn)}
 }
\maketitle

\begin{abstract}
In the undetermined linear system $\bm{b}=\mathcal{A}(\bm{X})+\bm{s}$, vector $\bm{b}$ and operator $\mathcal{A}$ are the known measurements and $\bm{s}$ is the unknown noise. In this paper, we investigate sufficient conditions for exactly reconstructing desired matrix  $\bm{X}$ being low-rank or approximately low-rank. We use the difference of nuclear norm and Frobenius norm ($L_{*-F}$) as a surrogate for rank function and establish a new nonconvex relaxation of such low rank matrix recovery, called the $L_{*-F}$ minimization, in order to approximate the rank function closer.   For such nonconvex and nonsmooth constrained $L_{*-F}$ minimization problems, based on whether the noise level is $0$, we give the upper bound estimation of the recovery error respectively. Particularly, in the noise-free case, one sufficient condition for exact recovery is presented. If linear operator $\mathcal{A}$ satisfies the restricted isometry property with $\delta_{4r}<\frac{\sqrt{2r}-1}{\sqrt{2r}-1+\sqrt{2}(\sqrt{2r}+1)}$, then $r$-\textbf{rank} matrix $\bm{X}$ can be exactly recovered without other assumptions. In addition, we also take insights into the regularized $L_{*-F}$ minimization model since such regularized model is more widely used in algorithm design. We provide the recovery error estimation of this regularized $L_{*-F}$ minimization model via RIP tool. To our knowledge, this is the first result on exact reconstruction of low rank matrix via regularized $L_{*-F}$ minimization.
\end{abstract}


%
\IEEEpeerreviewmaketitle

\section{Introduction}
%
%
%
%
\IEEEPARstart{L}{ow} rank matrix recovery (LMR) has been a rapidly growing filed of research in machine learning~\cite{chang2018unified}\cite{candes2010power}\cite{candes2012exact} and computer vision~\cite{candes2011robust}\cite{tomasi1993shape}. Mathematically, we hope to acquire the low rank matrix $\bm{X}^o$ satisfying:
\begin{equation}\label{no-noise}
    \bm{b}=\mathcal{A}(\bm{X}^o),
\end{equation}
where $\bm{b}$ is a given nonzero vector and $\mathcal{A}:\mathbb{R}^{m\times n}\rightarrow \mathbb{R}^l$ is a predesigned measurement linear operator. 
Hence LMR can be formulated as follows:
\begin{equation}\label{rank_pro}
    \begin{aligned}
        &\min_{\bm{X}}\ \textbf{rank}(\bm{X})\\
                &\bf{s.t.}\ \mathcal{A}(\bm{X})=\bm{b}.
    \end{aligned}
\end{equation}
 A particular class of~(\ref{rank_pro}) is to utilize a small number of observation entries to reconstruct matrix $\bm{X}$, referred as low rank matrix completion,  where $\mathcal{A}:=\mathcal{P}_{\Omega}$ the projection operator  samples entries from the index set $\Omega$. 
 Unfortunately, problem~(\ref{rank_pro}) is generally NP-hard and ill-posed~\cite{meka2008rank}, a famous convex surrogate function for rank function is the nuclear norm proposed by Fazel et al.~\cite{fazel2001rank} and they established a convex optimization problem over the same constraints: 
 \begin{equation}\label{mod_nuclear}
     \begin{aligned}
         &\min_{\bm{X}}\ \Vert \bm{X}\Vert_*\\
         &\textup{s.t.}\ \mathcal{A}(\bm{X})=\bm{b},
     \end{aligned}
 \end{equation}
where $\Vert \cdot\Vert_*$ equals the summation of  singular values. Note that the relaxation method is conceptually analogous to the relaxation from $
\ell_0$ norm to $\ell_1$ norm in compressing sensing~\cite{donoho2006compressed}. Many simple and computationally efficient optimization methods~\cite{cai2010singular}\cite{toh2010accelerated} exist to solve this type of nuclear norm  minimization problem. Variants of nuclear norm, including the truncated nuclear norm~\cite{zhang2012matrix} and weighted nuclear norm~\cite{gu2014weighted} were also proposed in the literature and enhance the recovery performance. Under some suitable conditions related to restricted
isometry property (RIP), problem~(\ref{mod_nuclear}) can be guaranteed to produce the the minimum-rank solution~\cite{recht2010guaranteed}. 

In this study, we focus on a non-convex surrogate for rank function, i.e., the difference of nuclear norm and Frobenius norm ($L_{*-F}$) to solve problem~(\ref{noise_AX}). Invoking the definition of Frobenius norm $\Vert \bm{X}\Vert_F:=\sqrt{\langle \bm{X}, \bm{X}\rangle}$, direct manipulations yield $\Vert \bm{X}\Vert_*-\Vert \bm{X}\Vert_F=\Vert \sigma(\bm{X})\Vert_1-\Vert \sigma(\bm{X})\Vert_2$. The contour plot of the $\Vert\cdot\Vert_1-\Vert\cdot\Vert_2$~($\ell_1-\ell_2$) metric presents in Figure \ref{fig:1}, 
\begin{figure}[t]
\centering
\subfigure{
 \includegraphics[width=2.5in]{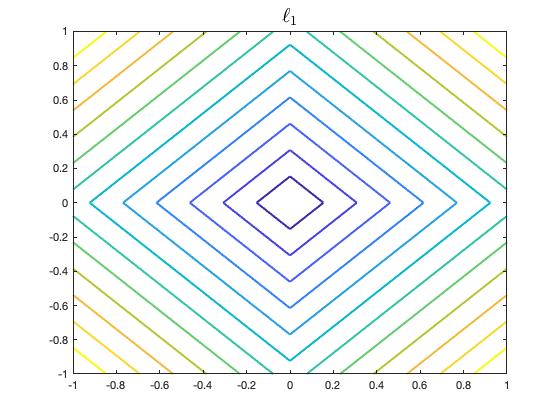}
}
\subfigure{
 \includegraphics[width=2.5in]{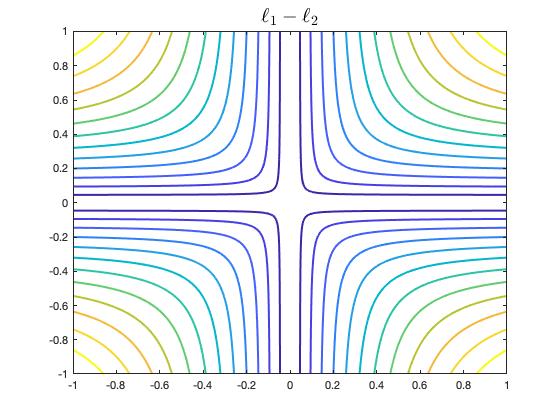}
}
\caption{\rm{From outside to inside, the value on the contour decreases gradually. It is notable that the solutions of $\ell_1-\ell_2$ 
approach two axes, when the value is close to $0$. Hence $\ell_1-\ell_2$ norm gets more sparse solutions than $\ell_1$ norm.}}\label{fig:1}
\end{figure}
which implies it can achieves the goal of sparsity and hence we can build a new nonconvex relaxation of the low rank matrix model (\ref{rank_pro}) as follows:
\begin{equation}
\min_{\bm{X}\in\mathbb{R}^{m\times n}}\Vert \bm{X}\Vert_*-\Vert \bm{X}\Vert_F\quad \textup{s.t.}\ \mathcal{A}(\bm{X})=\bm{b}.\label{exact_pro}
\end{equation}

In practice, since the measurement $\bm{b}$ is possibly contaminated by unknown noise $\bm{s}$, there produces a type of robustly recovering a low rank matrix in the form of 
\begin{equation}\label{noise_AX}
    \bm{b}=\mathcal{A}(\bm{X}^o)+\bm{s}.
\end{equation}
 Under this case, we can formulate as the following models, one minimizing the same function with~(\ref{exact_pro}) executes  the robust constraints to complete the low rank matrix recovery:
\begin{equation}
     \min_{\bm{X}\in\mathbb{R}^{m\times n}}\Vert \bm{X}\Vert_*-\Vert \bm{X}\Vert_F\quad \textup{s.t.}\ \Vert\mathcal{A}(\bm{X})-\bm{b}\Vert_2\leq\epsilon,\label{unexact_pro}
\end{equation}
where nonnegative parameter $\epsilon$ represents the noise level, and the other one is given by a sparsity regularized optimization problem:
\begin{equation}\label{unonstrainpro}
    \min_{\bm{X}\in\mathbb{R}^{m\times n}} \mathcal{J}(\bm{X}):= \Vert \bm{X}\Vert_*-\Vert \bm{X}\Vert_F+\frac{1}{2\lambda}\Vert\mathcal{A}(\bm{X})-b\Vert_2^2,
\end{equation}
where  $\lambda>0$ is a tradeoff hyperparameter. Both  problems (\ref{unexact_pro}) and (\ref{unonstrainpro}) belong to a special case of difference of convex function (DC) programming. For more details on DC programming and its algorithm implementation, see, e.g.,~\cite{tao1997convex}\cite{gong2013general}. 
In fact, no matter what function is chosen to replace the rank function, it is necessary to consider the recovery conditions and their resultant recovery error estimates. Subsequently, we will provide sufficient conditions to guarantee the robust reconstruction in bound of $\Vert\cdot\Vert_F$ or exact reconstruction of the desired low rank matrix $\bm{X}^o$ through above three minimization problems including noise setting ($\epsilon\neq0$) and noiseless context ($\epsilon=0$). 

For this purpose, One of the most commonly used tools  is the
RIP condition. The matrix version of RIP notion as defined in Definition~\ref{def_vec_rip}, which was first introduced by Cand$\grave{e}$s and Tao~\cite{candes2005decoding}, is widely used in sparse signal recovery~\cite{ge2020new}\cite{davenport2010analysis}\cite{chang2014improved}.\begin{definition}\label{def_vec_rip}
    For $T\subseteq\{1,\cdots,n\}$ and each number $r$, $r$-restricted isometry constants of matrix $\bm{A}$ is the smallest quantity $\delta_r$  such that  
    \begin{equation*}
        (1-\delta_r)\Vert \bm{x}\Vert_2^2\leq\Vert A_{T}\bm{x}\Vert^2_2\leq (1+\delta_r)\Vert \bm{x}\Vert_2^2
    \end{equation*}
    for all subsets $T$ with $\vert T\vert \leq r$ and all $\bm{x}\in\mathbb{R}^{\vert T\vert}$. The matrix $\bm{A}$ is said to satisfy the $r$-RIP with $\delta_r$.
\end{definition}
Inspired by this, Cand$\grave{e}$s and Plan~\cite{candes2010tight} introduced the isometry constants of a linear map $\mathcal{A}$, as defined in Definition~\ref{def_map_rip}. The linear map $\mathcal{A}$ is said to satisfy the RIP at rank $r$ if $\delta_r$ is bounded by a sufficiently small constant between $0$ and $1$. As they mentioned, fix $0\leq\delta<1$ and let $\mathcal{A}$ be a random measurement ensemble obeying the following condition: for any given $\bm{X}\in\mathbb{R}^{n_1\times n_2}$ and any fixed $0<\hat{t}<1$, $P(\vert \Vert\mathcal{A}(\bm{X})\Vert^2_2-\Vert \bm{X}\Vert^2_F\vert>\hat{t}\Vert \bm{X}\Vert_F^2)\leq c_1\exp(-c_2m)$ for fixed constants $c_1, c_2>0$. Then if $m\geq d_1nr$, $\mathcal{A}$ satisfies the RIP with isometry constant $\delta_r\leq \delta$ with probability exceeding $1-c_1e^{-d_2m}$ for fixed constants $n, d_1, d_2>0$. There is a rich literature providing a range of theoretical guarantees under which it is possible to recover a matrix based on the assumption that linear map $\mathcal{A}$ satisfies certain RIP conditions. See, e.g., ~\cite{cai2013sparse}\cite{tu2016low}\cite{bhojanapalli2016global}\cite{liu2022robust}. Many types of linear map, including random sensing designs~\cite{krahmer2011new}\cite{do2011fast}, are known to satisfy the RIP with high probability.

\begin{definition}\label{def_map_rip}
For each integer $r=1,2,\cdots,n$, the isometry constant $\delta_r$ of a linear map $\mathcal{A}$ is the smallest quantity such that
\begin{equation}\label{ripmap}
    (1-\delta_r)\Vert \bm{X}\Vert_F^2\leq\Vert\mathcal{A}(\bm{X})\Vert_2^2\leq(1+\delta_r)\Vert \bm{X}\Vert_F^2
\end{equation}
holds for all $r$-rank matrices (any matrix of rank no greater than $r$). 
\end{definition}
\subsection{Relation to Existing Work}
In the present study, some sufficient conditions based on the RIP analysis to guarantee the recovery of desired  matrices through the $L_{1-2}$ metric have been provided. Cai~\cite{cai2020minimization} gave a stably recovery condition on imposing an additional assumption on the dimension of desired matrix $\bm{X}^o$ besides the RIP condition. Let $\bm{X}^o$ satisfy~(\ref{noise_AX}) with $\Vert\bm{s}\Vert_2\leq\epsilon$, if there exists $r\leq\min\{m,n\}$ so that $\alpha(r):=(\frac{\sqrt{2r}-1}{\sqrt{r}+1})^2>1$ and the linear map $\mathcal{A}$ satisfies $\delta_{2r}+\alpha(r)\delta_{3r}<\alpha(r)-1$, the error estimation deriving from problem~(\ref{unexact_pro}) is bounded by $C_1\Vert \bm{X}^o-\bm{X}^o_r\Vert_F+C_2\epsilon$, where $C_1, C_2>0$ and $\bm{X}^o_r$ is the best rank-$r$ approximation matrix of $\bm{X}^o$. Hence in the noise-free case ($\epsilon=0$), $\bm{X}^o$ can be recovered exactly by the constrained optimization problem (\ref{exact_pro}), if $\bm{X}^o$ is $r$-rank matrix. Ma et al.~\cite{ma2017truncated} proposed a truncated $L_{1-2}$ metric and gave the theoretical guarantees to recovery the low rank matrix based on the corresponding constrained model. Note that $L_{1-2}$ metric is a special case for the truncated $L_{1-2}$ metric, hence following from ~\cite{ma2017truncated} it yields that for the $r$-rank matrix $\bm{X}^o$ satisfying (\ref{noise_AX}) with $\Vert\bm{s}\Vert_2\leq\epsilon$, the error estimation deriving from problem (\ref{unexact_pro}) is bounded by $C_3\epsilon$, where $C_3>0$. Hence in the noise-free case ($\epsilon=0$), $\bm{X}^o$ satisfying (\ref{no-noise}) can be recovered exactly by problem (\ref{exact_pro}). Above observation suggests that they can not provide a robust error estimation when the information about range of $\textbf{rank}(\bm{X}^o)$ is missing. There also exist other forms of characterizations for isometry constant of a linear map by replacing the vector $\ell_2$ norm with other vector norms, such as $\ell_p(0<p\leq 1)$ quasi norm~\cite{zhang2013restricted}. Under the framework of RIP with $p=1$, Guo et al.~\cite{guo2022low} presented a recovery guarantee through the truncated $L_{1-2}$ minimization and adopting $\ell_1$ constraints, naturally, the recovery estimation of $L_{1-2}$ minimization problem can be acquired. These works give the recovery estimation concerning $L_{1-2}$ minimization approach with constraints. In general, problem~(\ref{exact_pro}) and problem~(\ref{unexact_pro}) is not convenient to be solved in numerical implementation, the common choice is to solve a regularized variant~(\ref{unonstrainpro}) so that some algorithms for unconstrained minimization problem such as DCA~\cite{tao1997convex} and PG~\cite{boyd2004convex} can be adopted to complete the recovery of the low rank matrix~\cite{cai2020minimization}\cite{ma2017truncated}\cite{yao2016fast}. Hence it becomes very significant and necessary to develop some theoretical results for problem~(\ref{unonstrainpro}) and explore its relationship with constrained minimization problem. Moreover, since problem~(\ref{exact_pro}) is a surrogate optimization problem of NP hard problem~(\ref{rank_pro}), it is also necessary to establish the relation of optimal solutions among them. Motivated by above discussion, in this paper our contribution can be summarized as follows:
\begin{itemize}
    \item We update the recovery theory based on the RIP conditions of linear map $\mathcal{A}$ for constrained optimization problems~(\ref{exact_pro}) and~(\ref{unexact_pro}) to broaden the range of recoverable low rank matrix. Although authors in \cite{cai2020minimization} recently also give recovery theory based on the RIP conditions for these optimization problems, their recovery estimation is built for the desired matrix $\bm{X}^o$ satisfying $\min\{m,n\}\geq 24$ and ours breaks this restriction. Different from~\cite{ma2017truncated},  the recovery theory we proposed is still valid when the range of $\textbf{rank}(\bm{X}^o)$ is unclear.
    \item To the best of our knowledge, we are the first to provide the upper bound estimation of the approximate error for the regularized $L_{*-F}$ minimization problem~(\ref{unonstrainpro}). Actually, the existing theoretical investigation of problem~(\ref{unonstrainpro}) is limited to its induced algorithms and there is no theoretical guarantee of the regularized recovery estimation. We fill the blink of theoretical investigation to characterize its essential performance in robustly recovering the desired matrix $\bm{X}^o$ from~(\ref{noise_AX}). 
    \item The sufficient condition is provided to demonstrate that it is possible to recover the lowest rank solution exactly by minimizing the difference between nuclear norm and Frobenius norm. Moreover, we also briefly discuss the relation between the global minimizers of problems (\ref{unonstrainpro}) and (\ref{exact_pro}).
\end{itemize}
\subsection{Notation}
In this subsection, we introduce some related notations used throughout this paper. We present vectors by boldface lowcase letters, e.g., $\bm{a}$, matrices by boldface capital letters, e.g., $\bm{A}$, sets by capital letters, e.g., $A$, scalars by lowercase letters, e.g., $a$. For any positive integer $d$, $[d]$ denotes the index set $\{1,2,\cdots,d\}$. Given $\bm{X}\in\mathbb{R}^{m\times n}$. Define $t:=\min\{m,n\}$. Let $\sigma(\bm{X}):=(\sigma_1(\bm{X}),\cdots,\sigma_t(\bm{X}))$ be a vector composed of $\bm{X}$'s singular values with $\sigma_1(\bm{X})\geq\cdots\geq\sigma_t(\bm{X})\geq 0$. $\bm{E}\in\mathbb{R}^{t\times t}$ denotes the identity matrix. For an index set $S\subseteq [d]$, let $\vert S\vert$ denote its cardinality and  $S^c$ denote its complementarity set. Denote $\bm{X}_i$ as the $i$-th column of $\bm{X}$. Denote $\bm{X}_S$ as $\bm{X}$ with all but columns indexed by $S$ set to zero vector. For any vector $\bm{h}\in\mathbb{R}^n$ and any index subset $S\subseteq[n]$, we denote by $\bm{h}_S$ the vector whose entries $(\bm{h}_{S})_i=h_i$ for $i\in S$ and $0$ otherwise. Besides, we denote by $\bm{h}_{\max(k)}$ the vector $\bm{h}$ with all but the largest $k$ entries in absolute values set to $0$. The inner products of two vectors $\bm{x}, \bm{y}\in\mathbb{R}^n$ and two matrices $\bm{X}, \bm{Y}\in\mathbb{R}^{m\times n}$ are denoted by $\langle\bm{x},\bm{y}\rangle:=\bm{x}^{\top}\bm{y}$ and $\langle\bm{X},\bm{Y}\rangle:=\text{Tr}(\bm{X}^{\top}\bm{Y})$ where $\text{Tr}$ is the matrix trace. For any $\bm{x}\in\mathbb{R}^t$, we define the operator $\mathcal{D}:\mathbb{R}^{t}\rightarrow\mathbb{R}^{m\times n}$ as follows:
\begin{align*}
    \mathcal{D}_{ij}(\bm{x})=\left\{ 
    \begin{array}{cc}
       x_i\  &{\rm if}\ i=j,  \\
        0\  &{\rm otherwise}.
    \end{array}
    \right.
\end{align*}
Let $\mathcal{O}(n)$ be the group of $n\times n$ orthogonal matrices. For any given $\bm{X}\in\mathbb{R}^{m\times n}$, its singular value decomposition (SVD) is $\bm{X}=\bm{U}\mathcal{D}(\sigma(\bm{X}))\bm{V}^{\top}$ with $\bm{U}\in \mathcal{O}(m)$ and $V\in \mathcal{O}(n)$. Denote $\lfloor\cdot\rfloor$ and $\lceil\cdot\rceil$ as the symbols for floor function and ceiling function respectively. Let $\textup{supp}(\bm{x})$ denote the support of $\bm{x}$. We say two matrices $\bm{X}$ and $\bm{Y}$ in $\mathbb{R}^{m\times n}$ have simultaneous ordered SVD if there exist $\bm{U}_X, \bm{U}_Y\in\mathcal{O}^m$ and $\bm{V}_X, \bm{V}_Y\in\mathcal{O}^n$ such that $\bm{X}=\bm{U}_X\mathcal{D}(\sigma(\bm{X})){\bm{V}_X}^{\intercal}$ and $\bm{Y}=\bm{U}_Y\mathcal{D}(\sigma(\bm{Y}))\bm{V}_Y^{\intercal}$ with $\{{\bm{U}_X}_1, {\bm{U}_X}_2,\cdots, {\bm{U}_X}_m\}=\{{\bm{U}_Y}_1,{\bm{U}_Y}_2,\cdots, {\bm{U}_Y}_m\}$ and $\{{\bm{V}_X}_1,{\bm{V}_X}_2,\cdots,{\bm{V}_X}_n\}=\{{\bm{V}_Y}_1,{\bm{V}_Y}_2,\cdots,{\bm{V}_Y}_n\}$. That implies matrix $\bm{Y}$ can be rewritten as $\bm{U}_X\mathcal{D}(\bm{\pi}(\sigma(\bm{Y}))){\bm{V}_X}^{\intercal}$, where $\bm{\pi}_1(\sigma(\bm{Y})),\bm{\pi}_2(\sigma(\bm{Y})),\cdots,\bm{\pi}_t(\sigma(\bm{Y})) $ are some permutation of singular values of $\bm{Y}$. 

\section{Main Results on Constrained Minimization} \label{sec:learning_T}
In this section, we first show that it is possible to recover the lowest rank representation by solving a nonconvex optimization problem. And then, we present the recovery performance of $L_{1-2}$ minimization  model with constraints under the framework of RIP.
\subsection{Links Between Problem (\ref{rank_pro}) and Problem (\ref{exact_pro})}
The goal of this subsection is to study the links between the rank minimization problem (\ref{rank_pro}) and its nonconvex surrogate minimization (\ref{exact_pro}). In light of the characterization of locally sparse feasible solutions, we show the globally optimal solution of problem~(\ref{rank_pro}) must solve the problem~(\ref{exact_pro}) globally. 
\begin{definition}
Denote $\mathcal{F}:=\{\bm{X}\ \vert\ \mathcal{A}(\bm{X})=\bm{b}\}$. $X\in\mathcal{F}$ is called locally sparse if $\nexists Y\in\mathcal{F}\backslash\{\bm{X}\}$ such that $\bm{Y}$ and $\bm{X}$ have a simultaneous ordered SVD and $\textup{supp}(\bm{\pi}(\sigma(\bm{Y})))\subseteq \textup{supp}(\sigma(\bm{X}))$. Denote  $\mathcal{F}_L=\{\bm{X}\in\mathcal{F}\ \vert\  \bm{X} \textup{ is locally sparse}\}$ as the set of locally sparse feasible solutions.
\end{definition}
In fact, the locally sparse feasible solution is locally the sparsest feasible solution.
\begin{lemma}
    For any $\bm{X}\in\mathcal{F}_L$, there exists $\delta_X>0$ such that for any $\bm{Y}\in\mathcal{F}$ having simultaneous ordered SVD with $\bm{X} $, if $0<\Vert \bm{Y}-\bm{X} \Vert_F<\delta_X$, we have $\textup{supp}(\sigma(\bm{X} ))\subset \textup{supp}(\bm{\pi}(\sigma(\bm{Y})))$.
\end{lemma}
\begin{proof}
 Choose $\delta_X=\min_{i\in\textup{supp}(\sigma(\bm{X}))}\{\sigma_i(\bm{X})\}$. For any $\bm{Y}\in\mathcal{F}$ having simultaneous ordered SVD with $\bm{X} $ such that $0<\Vert \bm{Y}-\bm{X} \Vert_F<\delta_X$, we set $\bm{Z}=\bm{X} -\bm{Y}$, that is $\bm{Z}=\bm{U}_X\mathcal{D}(\sigma(\bm{X} )-\bm{\pi}(\sigma(\bm{Y}))){\bm{V}_X}^{\intercal}$. For brevity, denote $\bm{z}$ as $\sigma(\bm{X} )-\bm{\pi}(\sigma(\bm{Y}))$, then we get
 \begin{equation*}
     \Vert \bm{z}\Vert_\infty\leq  \Vert \bm{z}\Vert_2<\min_{i\in\textup{supp}(\sigma(\bm{X} ))}\{\sigma_i(\bm{X} )\}.
 \end{equation*}
 This yields 
 \begin{equation*}
\begin{aligned}
\bm{\pi}_i(\sigma(\bm{Y}))&\geq\sigma_i(\bm{X})-\Vert \bm{z}\Vert_\infty\\
&>\sigma_i(\bm{X} )-\min_{i\in\textup{supp}(\sigma(\bm{X} ))}\{\sigma_i(\bm{X} )\}\geq 0
\end{aligned}
 \end{equation*}
 for any $i\in\textup{supp}(\sigma(\bm{X} ))$, which implies $\textup{supp}(\sigma(\bm{X} ))\subseteq\textup{supp}(\bm{\pi}(\sigma(\bm{Y})))$. Moreover, since $\bm{X}\in\mathcal{F}_L$, we have $\textup{supp}(\sigma(\bm{X} ))\neq \textup{supp}(\bm{\pi}(\sigma(\bm{Y})))$. Hence we obtain that $\textup{supp}(\sigma(\bm{X} ))\subset\textup{supp}(\bm{\pi}(\sigma(\bm{Y})))$. This completes the proof.
 \end{proof}
 The following results show that the optimal solution sets of problem~(\ref{exact_pro}) and problem~(\ref{rank_pro}) are contained in the locally sparse sets.
 \begin{lemma}\label{pd_gl}
    If $\bm{X}^*$ solves problem~(\ref{exact_pro}) globally, then $\bm{X}^*\in\mathcal{F}_L$.
    \end{lemma}
\begin{proof}
    If $\bm{X}^*\notin\mathcal{F}_L$ and $\bm{X}^*=\bm{U}_{X^*}\mathcal{D}(\sigma(\bm{X}^*))\bm{V}_{X^*}^{\intercal}$, then there exists $Y^*\in\mathcal{F}\backslash\{\bm{X}^*\}$ such that $\bm{Y}^*=\bm{U}_{X^*}\mathcal{D}(\bm{\pi}(\sigma(\bm{Y}^*)))\bm{V}_{X^*}^{\intercal}$ and $\textup{supp}(\bm{\pi}(\sigma(\bm{Y})))\subseteq\textup{supp}(\sigma(\bm{X}^*))$. Hence we can find a small enough $\epsilon>0$ such that 
    \begin{equation*}
        \sigma(\bm{X}^*)-\epsilon\bm{\pi}(\sigma(\bm{Y}^*))\geq 0.
    \end{equation*}
    Define $\bm{Z}^*:=\bm{U}_{X^*}\mathcal{D}(\frac{\sigma(\bm{X}^*)-\epsilon\bm{\pi}(\sigma(\bm{Y}^*))}{1-\epsilon})\bm{V}_{X^*}^{\intercal}$. By directly calculating, it yields 
    \begin{equation*}
            \mathcal{A}(\bm{Z}^*)=\frac{1}{1-\epsilon}\mathcal{A}(\bm{X}^*)-\frac{\epsilon}{1-\epsilon}\mathcal{A}(\bm{Y}^*)=\bm{b}
    \end{equation*}
    and $\Vert \bm{Z}^*\Vert_*=\frac{1}{1-\epsilon}\Vert \bm{X}^*\Vert_*-\frac{\epsilon}{1-\epsilon}\Vert \bm{Y}^*\Vert_*$ due to the non-negativity of $\sigma_i(\bm{X}^*)-\epsilon\bm{\pi}_i(\sigma(\bm{Y}^*))$. Moreover, it follows from $\bm{Y}^*\neq \bm{X}^*$ and $\mathcal{A}(\bm{Y}^*)=\mathcal{A}(\bm{X}^*)$ that they are linearly independent, and hence $\bm{Y}^*$ and $\bm{Z}^*$ are linearly independent. This implies
    \begin{equation*}
        \Vert \bm{X}^*\Vert_F<\epsilon\Vert \bm{Y}^*\Vert_F+(1-\epsilon)\Vert \bm{Z}^*\Vert_F.
    \end{equation*}
Therefore, it is obvious that  
\begin{equation*}
\begin{aligned}
    &\Vert \bm{X}^*\Vert_*-\Vert \bm{X}\Vert_F\\
    >&\epsilon(\Vert \bm{Y}^*\Vert_*-\Vert \bm{Y}^*\Vert_F)+(1-\epsilon)(\Vert \bm{Z}^*\Vert_*-\Vert \bm{Z}^*\Vert_F)\\
    \geq&\min\{\Vert \bm{Y}^*\Vert_*-\Vert \bm{Y}^*\Vert_F, \Vert \bm{Z}^*\Vert_*-\Vert \bm{Z}^*\Vert_F\},
\end{aligned}
\end{equation*}
which contradicts with the optimality of $\bm{X}^*$.
\end{proof}
\begin{lemma}\label{p0_gl}
    If $\bm{X}^*$ solves problem~(\ref{rank_pro}) globally, then $\bm{X}^*\in\mathcal{F}_L$.
    \end{lemma}
    \begin{proof}
        If $\bm{X}^*\notin\mathcal{F}_L$ and $\bm{X}^*=\bm{U}_{X^*}\mathcal{D}(\sigma(X^*))\bm{V}_{X^*}^{\intercal}$, then there exists $\bm{Y}^*\in\mathcal{F}\backslash\{\bm{X}^*\}$ such that $\bm{Y}^*=\bm{U}_{X^*}\mathcal{D}(\bm{\pi}(\sigma(\bm{Y}^*)))\bm{V}_{X^*}^{\intercal}$ and $\textup{supp}(\bm{\pi}(\sigma(\bm{Y}^*)))\subseteq\textup{supp}(\sigma(\bm{X}^*))$. Since $\bm{X}^*$ is optimal, $\textup{supp}(\bm{\pi}(\sigma(\bm{Y}^*)))=\textup{supp}(\sigma(\bm{X}^*))$ and we denote such support set as $S$. 
        
        According to $\bm{X}^*\neq \bm{Y}^*$, it yields $\sigma(\bm{X}^*)\neq\bm{\pi}(\sigma(\bm{Y}^*))$ and hence $\min_{i\in S}\{\frac{\sigma_i(\bm{X}^*)}{\bm{\pi}_i(\sigma(\bm{Y}^*))}\}<1$ or $\min_{i\in S}\{\frac{\bm{\pi}_i(\sigma(\bm{Y}^*))}{\sigma_i(\bm{X}^*)}\}<1$ must true. Without loss of generality, let $\min_{i\in S}\{\frac{\sigma_i(\bm{X}^*)}{\bm{\pi}_i(\sigma(\bm{Y}^*))}\}=\frac{\sigma_k(\bm{X}^*)}{\bm{\pi}_k(\sigma(\bm{Y}^*))}=r<1$ for some $k\in S$. Then 
        \begin{equation*}
            \begin{aligned}
                \bm{Z}^*:=\frac{1}{1-r}\bm{X}^*-\frac{r}{1-r}\bm{Y}^*
                =\bm{U}_{X^*}\mathcal{D}(\frac{\sigma(\bm{X}^*)-r\bm{\pi}(\sigma(\bm{Y}^*))}{1-r})\bm{V}_{X^*}^{\intercal},
            \end{aligned}
        \end{equation*} 
which implies $\mathcal{A}(\bm{Z}^*)=\bm{b}$, that is, $\bm{Z}^*\in \mathcal{F}$. Moreover, denote $\bm{z}$ as $\frac{\sigma(\bm{X}^*)-r\bm{\pi}(\sigma(\bm{Y}^*))}{1-r}$, then $\bm{z}_k=0$ indicates $\textup{supp}(\sigma(\bm{Z}^*))\subsetneqq\textup{supp}(\sigma(\bm{X}^*))$, which contradicts with $\bm{X}^*$ being optimal solution of problem~(\ref{rank_pro}).
\end{proof}
Now, we are in the position to present one of our main recovery result.
     \begin{theorem}
        If $\bm{X}^*$ uniquely solves problem~(\ref{rank_pro}) with $\textup{rank}(\bm{X}^*)=s$ and $\min_{i\in \textup{supp}(\sigma(\bm{X}))}\sigma_i(\bm{X})>\frac{2(\sqrt{s}-1)}{\textup{rank}(\bm{X})-1}\Vert \bm{X}^*\Vert_F$ for any $\bm{X}\in\mathcal{F}_L\setminus\{\bm{X}^*\}$, then $\bm{X}^*$ also uniquely solves problem~(\ref{exact_pro}).
    \end{theorem}
     \begin{proof}
     First, we will show that for $\bm{X}\in\mathbb{R}^{m\times n}$, $(\textup{rank}(\bm{X})-1)/2\min_{i\in \textup{supp}(\sigma(\bm{X}))}\sigma_i(\bm{X})\leq \Vert \bm{X}\Vert_*-\Vert \bm{X}\Vert_F\leq (\sqrt{\textup{rank}(\bm{X})}-1)\Vert \bm{X}\Vert_F$. This upper bound can be immediately obtained from the Cauchy-Schwartz inequality. Next, we will give the lower bound. 
     
        By directly calculating, it yields 
        \begin{equation*}
        \begin{aligned}
            \Vert \bm{X}\Vert_*-\Vert \bm{X}\Vert_F&=\frac{\Vert \bm{X}\Vert_*^2-\Vert \bm{X}\Vert_F^2}{\Vert \bm{X}\Vert_*+\Vert \bm{X}\Vert_F}\\
            &=\frac{\sum_{i\neq j}\sigma_i(\bm{X})\sigma_j(\bm{X})}{\Vert \bm{X}\Vert_*+\Vert \bm{X}\Vert_F}\\
            &=\frac{\sum_{i\neq j\in\textup{supp}(\sigma(\bm{X}))}\sigma_i(\bm{X})\sigma_j(\bm{X})}{\Vert \bm{X}\Vert_*+\Vert \bm{X}\Vert_F}\\
            &\geq\frac{(\textup{rank}(\bm{X})-1)\Vert \bm{X}\Vert_*\min_{i\in \textup{supp}(\sigma(\bm{X}))}\sigma_i(\bm{X})}{\Vert \bm{X}\Vert_*+\Vert \bm{X}\Vert_F}\\
            &=\frac{\textup{rank}(\bm{X})-1}{1+\frac{\Vert \bm{X}\Vert_F}{\Vert \bm{X}\Vert_*}}\min_{i\in \textup{supp}(\sigma(\bm{X}))}\sigma_i(\bm{X})\\
            &\geq \frac{\textup{rank}(\bm{X})-1}{2}\min_{i\in \textup{supp}(\sigma(\bm{X}))}\sigma_i(\bm{X}).
            \end{aligned}
        \end{equation*}
       Considering above discussion, we can directly obtain 
        \begin{equation*}
        \begin{aligned}
            \Vert \bm{X}^*\Vert_*-\Vert \bm{X}^*\Vert_F&\leq (\sqrt{s}-1)\Vert \bm{X}^*\Vert_F\\
            &<\frac{\textup{rank}(\bm{X})-1}{2}\min_{i\in \textup{supp}(\sigma(\bm{X}))}\sigma_i(\bm{X})\\
            &\leq \Vert \bm{X}\Vert_*-\Vert \bm{X}\Vert_F
         \end{aligned}  
        \end{equation*}
        for any $\bm{X}\in\mathcal{F}_L\setminus\{\bm{X}^*\}$, which makes sense according to \Cref{p0_gl}. Thus, the desired result follows directly from \Cref{pd_gl}.
    \end{proof}

\subsection{Exact Recovery Theory}
In this subsection, we obtain some theoretical results to guarantee the robust recovery through the constrained $L_{1-2}$ minimization problem~(\ref{exact_pro}) and problem~(\ref{unexact_pro}). Our main results not only provide the sufficient conditions of stably recovering the desired matrix $X^o$, but also characterize the recovery errors with these two approaches.
Before proceeding, we provide essential preliminaries and related facts which are helpful to derive stable recovery conditions. We begin with the following fundamental properties respect to the  function $\Vert \bm{X}\Vert_*-\Vert \bm{X}\Vert_F$.

\begin{lemma}\label{lema_ineq_1-2}
    Suppose $\bm{X}\in\mathbb{R}^{m\times n}\backslash\{0\}$ and $\textbf{rank}(\bm{X})=r$. Let  $\bm{U}\mathcal{D}(\sigma(\bm{X}))\bm{V}^{\top}$ be the SVD of $\bm{X}$, that is, $\bm{X}=\sum_{i=1}^t\sigma_i(\bm{X})\bm{U}_i\bm{V}_i^{\top}$. Denote $\Lambda:=\textup{supp}(\sigma(\bm{X}))$. Then, we have 
    \begin{enumerate}[(a)]
     \item  $(t-\sqrt{t})\sigma_t(\bm{X})\leq \Vert \bm{X}\Vert_*-\Vert \bm{X}\Vert_F\leq(\sqrt{t}-1)\Vert \bm{X}\Vert_F;$
        \item $(r-\sqrt{r})\sigma_r(\bm{X})\leq \Vert \bm{X}\Vert_*-\Vert \bm{X}\Vert_F\leq(\sqrt{r}-1)\Vert \bm{X}\Vert_F;$
       \item $\Vert \bm{X}\Vert_*-\Vert \bm{X}\Vert_F=0$ if and only if $r=1$.
    \end{enumerate}
\end{lemma}
\begin{proof}
(a) We can easily get the supper bound from the Cauchy-Schwartz  inequality, and it suffices to give the lower bound. Define $r_0=\lfloor\sqrt{t}\rfloor$. We will give 
\begin{equation}\label{norm_F_inequ}
\begin{aligned}
    \Vert \bm{X}\Vert_F=&\sqrt{\sum_i\sigma_i(\bm{X})^2}
    \\\leq& \sum_{i=1}^{r_0}\sigma_i(\bm{X})+(\sqrt{t}-r_0)\sigma_{r_0+1}(\bm{X}).
    \end{aligned}
\end{equation}
To achieve this, let the left side of (\ref{norm_F_inequ}) be $\rho_l$ and the right side be $\rho_r$. On the one hand, we have
\begin{equation*}
\begin{aligned}
    \rho_l^2&=\sum_{i=1}^{r_0}\sigma_i(\bm{X})^2+\sum_{i=r_0+1}^{t}\sigma_i(\bm{X})^2\\
    &\leq \sum_{i=1}^{r_0}\sigma_i(\bm{X})^2+(t-r_0)\sigma_{r_0+1}(\bm{X})^2
\end{aligned}
\end{equation*}
and on the other hand by directly calculating, we obtain 
\begin{equation*}
\begin{aligned}
    \rho_r^2 
    =& \sum_{i=1}^{r_0}\sigma_i(\bm{X})^2+2(\sqrt{t}-r_0)\sigma_{r_0+1}(\bm{X})\sum_{i=1}^{r_0}\sigma_{i}(\bm{X})+\sum_{i=1}^{r_0}\sum_{\stackrel{j=1}{j\neq i}}^{r_0}\sigma_{i}(\bm{X})\sigma_{j}(\bm{X})
    +(\sqrt{t}-r_0)^2\sigma_{r_0+1}(\bm{X})^2\\
    \geq& \sum_{i=1}^{r_0}\sigma_i(\bm{X})^2+2(\sqrt{t}-r_0)\sigma_{r_0+1}(\bm{X})\sum_{i=1}^{r_0}\sigma_{r_0+1}(\bm{X})+\sum_{i=1}^{r_0}\sum_{\stackrel{j=1}{j\neq i}}^{r_0}\sigma_{r_0+1}(\bm{X})\sigma_{r_0+1}(\bm{X})
    +(\sqrt{t}-r_0)^2\sigma_{r_0+1}(\bm{X})^2\\
   =&\sum_{i=1}^{r_0}\sigma_i(\bm{X})^2+(t-r_0)\sigma_{r_0+1}(\bm{X})^2,
\end{aligned}
\end{equation*}
and hence $\rho_l\leq \rho_r$, which implies~(\ref{norm_F_inequ}).
Then it yields that  
\begin{equation*}
\begin{aligned}
    \Vert \bm{X}\Vert_*-\Vert \bm{X}\Vert_F
    \geq &\Vert \bm{X}\Vert_*- \sum_{i=1}^{r_0}\sigma_i(\bm{X})-(\sqrt{t}-r_0)\sigma_{r_0+1}(\bm{X})\\
    \geq &\sum_{i=r_0+1}^t \sigma_{i}(\bm{X})-(\sqrt{t}-r_0)\sigma_{r_0+1}(\bm{X})\\
    \geq & (t-\sqrt{t})\sigma_{t}(\bm{X}).
\end{aligned}
\end{equation*}
This yields the desired results.

(b) Note that $\vert\Lambda\vert=r$. Define $\hat{\bm{X}}:=\sum_{i\in\Lambda}\sigma_i(\bm{X})\bm{U}_i\bm{V}_i^{\top}$. Obviously, 
\begin{equation*}
    \Vert \bm{X}\Vert_*-\Vert \bm{X}\Vert_F=\Vert \hat{\bm{X}}\Vert_*-\Vert \hat{\bm{X}}\Vert_F,
\end{equation*}
hence we can apply (a) to get the desired results.

(c) If $\Vert \bm{X}\Vert_*-\Vert \bm{X}\Vert_F=0$, we can obtain $(r-\sqrt{r})\sigma_r(\bm{X})=0$ by employing the relation (a) and hence $r=1$. The other direction is easy.
\end{proof}
By a simple application of the parallelogram identity, we then have the next lemma, whose proof can be found in \cite{candes2011tight}. 
\begin{lemma}\label{inner_0}
    For all $\bm{X}, \bm{X}'$ obeying $\langle \bm{X},\bm{X}'\rangle=0$, and $\textbf{rank}(\bm{X})\leq r$, $\textbf{rank}(\bm{X}')\leq r'$, we have
    \begin{equation*}
        \vert\langle \mathcal{A}(\bm{X}),\mathcal{A}(\bm{X}')\rangle\vert\leq\delta_{r+r'}\Vert \bm{X}\Vert_F\Vert \bm{X}'\Vert_F,
    \end{equation*}
where $\delta_{r+r'}$ is $(r+r')$-isometry constant defined by Definition \ref{def_map_rip}.
\end{lemma}
To show the main results, the following lemma is also necessary.
\begin{lemma}\label{mn_ineq}
Fix positive integer $r$ and $\hat{t}$. Let $m_1,m_2,n_1,n_2$ be nonnegative integers which satisfy $m_1+m_2=n_1+n_2=r$. Then for $k\in[\hat{t}]$, it holds
\begin{equation}
\begin{aligned}
  \min_{m_1,m_2,n_1,n_2}\max\{m_1+n_1+k,m_2+n_2+2k\}
=\max\{r+\lceil\frac{3}{2}k\rceil, 2k\}.  
\end{aligned}
\end{equation}
\end{lemma}
\begin{proof}
    Obviously, it holds  
    \begin{equation}
    \max\{m_1+n_1+k,m_2+n_2+2k\}\geq r+ \lceil\frac{3}{2}k\rceil 
    \end{equation}
    for any $k\in [\hat{t}]$ since $m_1+n_1+k+n_1+n_2+2k=2r+3k$. 

    For $k\geq 2r$. Following from $m_1+n_1\leq 2r$, it is known that $\max\{m_1+n_1+k,m_2+n_2+2k\}=m_2+n_2+2k$ and then we have $\min_{m_1,m_2,n_1,n_2}\max\{m_1+n_1+k,m_2+n_2+2k\}=2k$ by setting $m_2=n_2=0$.
    
    For $k< 2r$. It is known that $r+\lceil\frac{3}{2}k\rceil\geq r+\frac{3}{2}k>2k$. In fact, we can show $\max\{m_1+n_1+k,m_2+n_2+2k\}= r+ \lceil\frac{3}{2}k\rceil$ by setting an appropriate $m_1, m_2, n_1, n_2$. If $k$ is odd, we set 
    \begin{equation*}
    \begin{aligned}
        &m_1=\lceil\frac{r}{2}+\frac{k}{4}+\frac{1}{4}\rceil;\\
        &m_2=r-\lceil\frac{r}{2}+\frac{k}{4}+\frac{1}{4}\rceil;\\
        &n_1=r+\frac{k}{2}+\frac{1}{2}-\lceil\frac{r}{2}+\frac{k}{4}+\frac{1}{4}\rceil;\\
        &n_2=\lceil\frac{r}{2}+\frac{k}{4}+\frac{1}{4}\rceil-\frac{k}{2}-\frac{1}{2}.
    \end{aligned}
    \end{equation*}
    If $k$ is even, we set 
    \begin{equation*}
        \begin{aligned}
            &m_1=\lceil\frac{r}{2}+\frac{k}{4}\rceil;\\
            &m_2=r-\lceil\frac{r}{2}+\frac{k}{4}\rceil;\\
            &n_1=r+\frac{1}{2}k-\lceil\frac{r}{2}+\frac{k}{4}\rceil;\\
            &n_2=\lceil\frac{r}{2}+\frac{k}{4}\rceil-\frac{k}{2}.
        \end{aligned}
    \end{equation*}
    This indicates that 
    \begin{equation*}
      \min_{m_1,m_2,n_1,n_2}\max\{m_1+n_1+k,m_2+n_2+2k\}= r+ \lceil\frac{3}{2}k\rceil. 
    \end{equation*}
    Combining the above two cases, we obtain the desired results.
\end{proof}
In view of above lemmas, we are ready to give our main results for recovery analysis via constrained  $L_{1-2}$ minimization problem. We first present the recovery estimation in the noise-free case, i.e.,  the vector $\bm{s}=0$ in~(\ref{noise_AX}). The proof is based on the block
decomposition of a matrix and the technique result of
Lemma~\ref{mn_ineq}. Denote $\bm{U}\mathcal{D}(\sigma(\bm{X}^o))\bm{V}^{\top}$ as SVD of $\bm{X}^o$. For any fixed positive integer $r\leq t$, the best rank-r approximation matrix $\bm{X}^o_r$ of $\bm{X}^o$ is defined as $\bm{U}\mathcal{D}(\bm{x}^r)\bm{V}^{\top}$ where $\bm{x}^r\in\mathbb{R}^t$ and $x^r_i=\sigma_i(\bm{X}^o)$ for $i\in[r]$ and $x^r_i=0$ otherwise.
\begin{theorem}\label{rip_exact}
   Let $\mathcal{A}:\mathbb{R}^{m\times n}\rightarrow\mathbb{R}^l$ be a linear map and $\bm{b}\in\mathbb{R}^l$. Set the desired matrix $\bm{X}^o$ with $\mathcal{A}(\bm{X}^o)=\bm{b}$. For a given positive integer $r\leq t$, define $\hat{t}=\min\{m-r, n-r\}$. If there exists positive integer $k\in[\hat{t}]\setminus\{1\}$ such that linear map $\mathcal{A}$ obeys $\delta_{2r+k}<1$ and
    \begin{equation}
      \beta:= \cfrac{\sqrt{2}\delta_{\max\{r+\lceil\frac{3}{2}k\rceil,2k\}}(\sqrt{2r}+1)}{(1-\delta_{2r+k})(\sqrt{k}-1)}<1,
    \end{equation}
 then we obtain that any optimal solution $\hat{\bm{X}}^*$ of problem~(\ref{exact_pro}) satisfies
 \begin{equation*}
     \Vert \bm{X}^o-\hat{\bm{X}}^*\Vert_F\leq \alpha\Vert \bm{X}^o-X_r^o\Vert_F,
 \end{equation*}
 where $\alpha:=\cfrac{2(\sqrt{2r}+1)+2\beta(\sqrt{k}-1)}{(\sqrt{k}-1)(1-\beta)(\sqrt{2r}+1)}$. Moreover, if $\bm{X}^o$ is $r$-rank matrix, then the unique minimizer of problem~(\ref{exact_pro}) is exactly $\bm{X}^o$.
\end{theorem}
\begin{proof}
Define $\bm{Z}:=\hat{\bm{X}}^*-\bm{X}^o$ and $\bm{X}^o_{r_c}:=\bm{U}\mathcal{D}(\bm{x}^{r_c})\bm{V}^{\top}$ where $x^{r_c}_i=\sigma_i(\bm{X}^o)$ for $i\in[t]\backslash[r]$ and $x^{r_c}_i=0$ otherwise. Obviously, $\bm{X}^o=\bm{X}^o_{r}+\bm{X}^o_{r_c}$. 
Fix any nonnegative integer $m_i, n_i$ ($i\in[3]$) satisfying $m_1+m_2=n_1+n_2=r$, $m_1+m_2+m_3=m$ and $n_1+n_2+n_3=n$. Denote
\begin{equation*}
r_1:=m_1+n_1+k,~~ r_2:=m_2+n_2+2k,~~ r_3=\max\{r_1,r_2\}.
\end{equation*}

In the sequel, we get a block decomposition of $\bm{Z}$ with respect to $\bm{X}^o$ as follows:
\begin{equation}
    \bm{U}^{\top}\bm{Z} \bm{V}:=
 \begin{pmatrix}
        \bm{Z}_{11} & \bm{Z}_{12} &\bm{Z}_{13} \\
        \bm{Z}_{21} & \bm{Z}_{22} &\bm{Z}_{23} \\
        \bm{Z}_{31} & \bm{Z}_{32} &\bm{Z}_{33}
\end{pmatrix}
\end{equation}
with $\bm{Z}_{ij}\in\mathbb{R}^{m_j\times n_j}$ ($i,j\in[3]$). 
Define 
 \begin{equation*}
 \bm{Z}_1:=\bm{U}\left(
     \begin{array}{ccc}
     \bm{Z}_{11} & \bm{Z}_{12} & \bm{Z}_{13}\\
     \bm{Z}_{21} & 0 & 0\\
     \bm{Z}_{31} & 0 & 0 
     \end{array}
     \right)
     \bm{V}^{\top},\quad
     \bm{Z}_2:=\bm{U}\left(
     \begin{array}{ccc}
     0 & 0 & 0\\
     0 & \bm{Z}_{22} & \bm{Z}_{23}\\
     0 & \bm{Z}_{32} & 0 
     \end{array}
     \right)
     \bm{V}^{\top},\quad
        \bm{Z}_3:=\bm{U}\left(
     \begin{array}{ccc}
     0 & 0 & 0\\
     0 & 0 & 0\\
     0 & 0 & \bm{Z}_{33} 
     \end{array}
     \right)
     \bm{V}^{\top}.
 \end{equation*}
Naturally, we can decompose $\bm{Z}$ as 
\begin{equation}
    \bm{Z}=\bm{Z}_1+\bm{Z}_2+\bm{Z}_3=\bm{Z}^r+\bm{Z}^{r_c},
\end{equation}
 where $\bm{Z}^r:=\bm{Z}_1+\bm{Z}_2$, $\bm{Z}^{r_c}:=\bm{Z}_3$. 
  It is known that $\textbf{rank}(\bm{Z}_1)\leq m_1+n_1$, $\textbf{rank}(\bm{Z}_2)\leq m_2+n_2$ and $\textbf{rank}(\bm{Z}_1+\bm{Z}_2)\leq 2r$. Also, $\bm{Z}_1$, $\bm{Z}_2$ and $\bm{Z}_3$ are orthogonal each other. Then, we have 
\begin{equation*}
\begin{aligned} 
   &\Vert \bm{X}^o+\bm{Z}\Vert_*-\Vert \bm{X}^o+\bm{Z}\Vert_F\\ 
   =&\Vert \bm{X}^o_r+\bm{X}^o_{r_c}+\bm{Z}^r+\bm{Z}^{r_c}\Vert_*-\Vert \bm{X}^o+\bm{Z}^r+\bm{Z}^{r_c}\Vert_F\\
   \geq& \Vert \bm{X}^o_r+\bm{Z}^{r_c}\Vert_*-\Vert \bm{X}^o_{r_c}\Vert_*-\Vert \bm{Z}^r\Vert_*-\Vert \bm{X}^o\Vert_F
   -\Vert \bm{Z}^r\Vert_F-\Vert \bm{Z}^{r_c}\Vert_F\\
   =& \Vert \bm{X}^o_r\Vert_*+\Vert \bm{Z}^{r_c}\Vert_*-\Vert \bm{X}^o_{r_c}\Vert_*-\Vert \bm{Z}^r\Vert_*-\Vert \bm{X}^o\Vert_F-\Vert \bm{Z}^r\Vert_F-\Vert \bm{Z}^{r_c}\Vert_F.
\end{aligned}
\end{equation*}
Here the last equality follows from \cite[Lemma 2.3]{recht2010guaranteed} with $\bm{X}^o_r {\bm{Z}^{r_c}}^{\top}=0$ and ${\bm{X}^o_r}^{\top}\bm{Z}^{r_c}=0$. Besides, we can get 
\begin{equation*}
    \Vert \hat{\bm{X}}^*\Vert_*-\Vert \hat{\bm{X}}^*\Vert_F\leq \Vert \bm{X}^o\Vert_*-\Vert \bm{X}^o\Vert_F
\end{equation*}
due to the optimality of $\hat{\bm{X}}^*$. Hence we  obtain 
\begin{equation}\label{zrc_ineq}
    \Vert \bm{Z}^{r_c}\Vert_*-\Vert \bm{Z}^{r_c}\Vert_F\leq \Vert \bm{Z}^r\Vert_*+\Vert \bm{Z}^r\Vert_F+2\Vert \bm{X}^o_{r_c}\Vert_*.
\end{equation}
Denote $\bm{P}\mathcal{D}(\sigma(\bm{Z}_{33}))\bm{Q}^{\top}$ as the SVD of $\bm{Z}_{33}$ where $\sigma(\bm{Z}_{33})\in\mathbb{R}^{\hat{t}}$ and $\bm{P}\in\mathbb{R}^{(m-r)\times(m-r)}$, $\bm{Q}\in\mathbb{R}^{(n-r)\times(n-r)}$ are orthogonal matrices. We begin dividing $[\hat{t}]$ into subsets of size $k (1< k\leq \hat{t})$, that is 
\begin{equation*}
    [\hat{t}]=T_1\cup T_2\cup\cdots\cup T_h,
\end{equation*}
 where each $T_i$ contains $k$ indices probably except $T_h$, and $T_1$ contains the indices of $k$ largest coefficients of $\sigma(\bm{Z}_{33})$, $T_2$ contains the indices of  the next $k$ largest coefficients, and so on. Define 
 \begin{equation}
      \bm{Z}_{T_i}:=\bm{U}\left(
 \begin{array}{ccc}
    0  &  0 & 0\\
    0  &  0 & 0\\
    0  &  0 & \bm{P}\mathcal{D}(\sigma_{T_i}(\bm{Z}_{33}))\bm{Q}^{\top}
 \end{array}
\right)\bm{V}^{\top}
 \end{equation}
for $1\leq i\leq h$. Hence $\bm{Z}_1$, $\bm{Z}_2$, $\bm{Z}_{T_i}$ are all orthogonal to each other and $\textbf{rank}(\bm{Z}_{T_i})\leq k$. Then, on the one hand, we obtain 
\begin{equation*}
\begin{aligned}
    &\Vert\mathcal{A}(\bm{Z}_1+\bm{Z}_2+\bm{Z}_{T_1})\Vert_2^2\\
    =&\langle\mathcal{A}(\bm{Z}_1+\bm{Z}_2+\bm{Z}_{T_1}),\mathcal{A}(\bm{Z}_1+\bm{Z}_2+\bm{Z}_{T_1})\rangle\\
    =&\langle\mathcal{A}(\bm{Z}_1+\bm{Z}_2+\bm{Z}_{T_1}),\mathcal{A}(\bm{Z})\rangle-\langle\mathcal{A}(\bm{Z}_1+\bm{Z}_2
    +\bm{Z}_{T_1}),\mathcal{A}(\bm{Z}_3-\bm{Z}_{T_1})\rangle\\
    =&-\langle\mathcal{A}(\bm{Z}_1+\bm{Z}_2+\bm{Z}_{T_1}),\mathcal{A}(\bm{Z}_3-\bm{Z}_{T_1})\rangle.
\end{aligned}
\end{equation*}
On the other hand, direct calculation yields 
\begin{align}
    &\vert \langle\mathcal{A}(\bm{Z}_1+\bm{Z}_2+\bm{Z}_{T_1}),\mathcal{A}(\bm{Z}_3-\bm{Z}_{T_1})\rangle\vert\notag\\
    =&\vert \langle\mathcal{A}(\bm{Z}_1),\mathcal{A}(\bm{Z}_3-\bm{Z}_{T_1})\rangle+\langle\mathcal{A}(\bm{Z}_2+\bm{Z}_{T_1}),\mathcal{A}(\bm{Z}_3-\bm{Z}_{T_1})\rangle\vert\notag\\
    \leq&\sum_{i\geq 2}\vert\langle\mathcal{A}(\bm{Z}_1),\mathcal{A}(\bm{Z}_{T_i})\rangle\vert+\sum_{i\geq 2}\vert\langle\mathcal{A}(\bm{Z}_2+\bm{Z}_{T_1}),\mathcal{A}(\bm{Z}_{T_i})\rangle\vert\notag\\   
    \leq& (\delta_{r_1}\Vert \bm{Z}_1\Vert_F+\delta_{r_2}\Vert \bm{Z}_2+\bm{Z}_{T_1}\Vert_F)\sum_{i\geq 2}\Vert \bm{Z}_{T_i}\Vert_F\notag\\
    \leq&\sqrt{2} \delta_{r_3}\Vert \bm{Z}_1+\bm{Z}_2+\bm{Z}_{T_1}\Vert_F\sum_{i\geq 2}\Vert \bm{Z}_{T_i}\Vert_F\notag,
\end{align}    
where the second inequality comes from Lemma~\ref{inner_0} and the third inequality comes from the monotonicity of the $RIP$ constant.
 Easily, invoking Lemma~\ref{mn_ineq}, we can find that  
\begin{equation*}
\begin{aligned}
&\vert \langle\mathcal{A}(\bm{Z}_1+\bm{Z}_2+\bm{Z}_{T_1}),\mathcal{A}(\bm{Z}_3-\bm{Z}_{T_1})\rangle\vert\\
\leq& \sqrt{2} \delta_{\max\{r+\lceil\frac{3}{2}k\rceil,2k\}}\Vert \bm{Z}_1+\bm{Z}_2+\bm{Z}_{T_1}\Vert_F\sum_{i\geq 2}\Vert \bm{Z}_{T_i}\Vert_F
\end{aligned}
\end{equation*}
since the arbitrariness of $m_1,n_1,m_2,n_2$.

Hence 
\begin{equation*}
\begin{aligned}
    &\Vert\mathcal{A}(\bm{Z}_1+\bm{Z}_2+\bm{Z}_{T_1})\Vert_2^2
    \\\leq&  \sqrt{2} \delta_{\max\{r+\lceil\frac{3}{2}k\rceil,2k\}}\Vert \bm{Z}_1+\bm{Z}_2+\bm{Z}_{T_1}\Vert_F\sum_{i\geq 2}\Vert \bm{Z}_{T_i}\Vert_F,
\end{aligned}
\end{equation*}
combining with the fact that $$\Vert\mathcal{A}(\bm{Z}_1+\bm{Z}_2+\bm{Z}_{T_1})\Vert_2^2\geq (1-\delta_{2r+k})\Vert \bm{Z}_1+\bm{Z}_2+\bm{Z}_{T_1}\Vert_F^2,$$ we have 
\begin{equation}\label{z12t_ineq}
    \Vert \bm{Z}_1+\bm{Z}_2+\bm{Z}_{T_1}\Vert_F\leq\frac{ \sqrt{2} \delta_{\max\{r+\lceil\frac{3}{2}k\rceil,2k\}}}{1-\delta_{2r+k}}\sum_{i\geq 2}\Vert \bm{Z}_{T_i}\Vert_F.
\end{equation}

Now we will give an upper bound for the right side of (\ref{z12t_ineq}). 
For any $q\in T_i$ with $i\geq 2$, according to Lemma~\ref{lema_ineq_1-2} it yields 
\begin{equation*}
    \sigma_q\leq \min_{p\in T_{i-1}}\sigma_p(\bm{Z}_{33})\leq\cfrac{\Vert \bm{Z}_{T_{i-1}}\Vert_*-\Vert \bm{Z}_{T_{i-1}}\Vert_F}{k-\sqrt{k}},
\end{equation*}
which implies 
\begin{equation*}
    \Vert \sigma_{T_{i}}(\bm{Z}_{33})\Vert_2\leq \sqrt{k}\cfrac{\Vert \bm{Z}_{T_{i-1}}\Vert_*-\Vert \bm{Z}_{T_{i-1}}\Vert_F}{k-\sqrt{k}}.
\end{equation*}
By direct calculation, we have 
\begin{align}\label{sum_zj}
    \sum_{i=2}^h\Vert \bm{Z}_{T_i}\Vert_F \nonumber
    \leq& \frac{1}{\sqrt{k}-1}\sum_{i=2}^h(\Vert \bm{Z}_{T_{i-1}}\Vert_*-\Vert \bm{Z}_{T_{i-1}}\Vert_F)\nonumber\\
    \leq& \frac{1}{\sqrt{k}-1}\sum_{i=1}^h(\Vert \bm{Z}_{T_{i}}\Vert_*-\Vert \bm{Z}_{T_{i}}\Vert_F)\nonumber\\
    \leq& \frac{1}{\sqrt{k}-1}(\Vert \bm{Z}_3\Vert_*-\Vert \bm{Z}_3\Vert_F)\nonumber\\
   \leq& \frac{\sqrt{2r}+1}{\sqrt{k}-1}\Vert \bm{Z}_1+\bm{Z}_2\Vert_F+\frac{2}{\sqrt{k}-1}\Vert \bm{X}^o_{r_c}\Vert_*,
\end{align}
where the third inequality holds since $$\sum_{i=1}^h\Vert \bm{Z}_{T_{i}}\Vert_F\geq\sqrt{\sum_{i=1}^h\Vert \bm{Z}_{T_{i}}\Vert_F^2}=\Vert \bm{Z}_3\Vert_F,$$ and the last inequality comes from~(\ref{zrc_ineq}). 
Combining with~(\ref{z12t_ineq}), it is easy to verify that 
\begin{equation}\label{sumz12t1}
  \Vert \bm{Z}_1+\bm{Z}_2+\bm{Z}_{T_1}\Vert_F\leq \frac{2\beta}{(\sqrt{2r}+1)(1-\beta)}\Vert \bm{X}^o_{r_c}\Vert_*
\end{equation}
since $1-\beta>0$ from the assumption. 
Then we have
\begin{equation*}
    \begin{aligned}
        \Vert \bm{Z}\Vert_F&\leq \Vert \bm{Z}_1+\bm{Z}_2+\bm{Z}_{T_1}\Vert_F+\sum_{i=2}^h\Vert \bm{Z}_{T_i}\Vert_F\\
        &\leq \frac{\sqrt{k}+\sqrt{2r}}{\sqrt{k}-1}\Vert \bm{Z}_1+\bm{Z}_2+\bm{Z}_{T_1}\Vert_F+\frac{2}{\sqrt{k}-1}\Vert \bm{X}^o_{r_c}\Vert_*\\
        &\leq \alpha \Vert \bm{X}^o_{r_c}\Vert_*,
    \end{aligned}
\end{equation*}
where the first inequality follows from~(\ref{sum_zj}) and the last inequality follows from~(\ref{sumz12t1}).
This indicates 
\begin{equation*}
  \Vert \hat{\bm{X}}^*-\bm{X}^o \Vert_F \leq \alpha \Vert \bm{X}^o_{r_c}\Vert_*.
\end{equation*}
Hence, we get the desired results. The special case $\textbf{rank}(\bm{X}^o)\leq r$ implies $\bm{X}^o=\bm{X}^o_r$ and hence $\hat{\bm{X}}^*=\bm{X}^o$ is trivial. Thus, we complete the proof.
\end{proof}
Naturally, by setting different values of $k$, we can get different bounds on isometry constant. When choosing $k=2r$, we obtain the following RIP condition involved in $\delta_{4r}$. 
\begin{corollary}
Let $\mathcal{A}:\mathbb{R}^{m\times n}\rightarrow\mathbb{R}^l$ be a linear map and $\bm{b}\in\mathbb{R}^l$. Set the desired matrix $\bm{X}^o$ satisfies $\mathcal{A}(\bm{X}^o)=\bm{b}$, that implies the vector $\bm{s}=0$ in~(\ref{noise_AX}). For a given positive integer $r\leq t$, if linear map $\mathcal{A}$ obeys 
    \begin{equation}
      \delta_{4r}<\cfrac{\sqrt{2r}-1}{\sqrt{2r}-1+\sqrt{2}(\sqrt{2r}+1)},
    \end{equation}
 then any optimal solution $\hat{\bm{X}}^*$ of problem~(\ref{exact_pro}) satisfies
 \begin{equation*}
     \Vert \bm{X}^o-\hat{\bm{X}}^*\Vert_F\leq \hat{\alpha}\Vert \bm{X}^o-X_r^o\Vert_F,
 \end{equation*}
 where $\hat{\alpha}:=\cfrac{2+(2\sqrt{2}-2)\delta_{4r}}{(\sqrt{2r}-1)-[(\sqrt{2r}-1)+\sqrt{2}(\sqrt{2r}+1)]\delta_{4r}}$. Moreover, if $\bm{X}^o$ is $r$-rank matrix, the unique minimizer of problem~(\ref{exact_pro}) is exactly $\bm{X}^o$.
\end{corollary}
Next we shall establish the recovery estimation of constrained $L_{*-F}$ minimization when measurements are contaminated by noise.
\begin{theorem}
Under the assumptions of Theorem~\ref{rip_exact} except that the desired matrix $\bm{X}^o$ satisfies
     $\mathcal{A}(\bm{X}^o)+\bm{s}=\bm{b}$ with perturbation $\Vert\bm{s}
     \Vert_2\leq \epsilon$, we have that the optimal solution $\bm{X}^{opt}$ of problem~(\ref{unexact_pro})
 satisfies
 \begin{equation*}
     \Vert \bm{X}^o-\bm{X}^{opt}\Vert_F\leq \alpha\Vert \bm{X}^o-X_r^o\Vert_F+\bar{\alpha}\epsilon,
 \end{equation*}
 where 
   $$\alpha:=\cfrac{2(\sqrt{2r}+1)+2\beta(\sqrt{k}-1)}{(\sqrt{k}-1)(1-\beta)(\sqrt{2r}+1)},$$ and  $$\bar{\alpha}:=\cfrac{2(\sqrt{k}+\sqrt{2r})\sqrt{1+\delta_{2r+k}}}{(\sqrt{k}-1)(1-\beta)(1-\delta_{2r+k})}.$$
 Moreover, if $\bm{X}^o$ is $r$-rank matrix, then it yields 
 \begin{equation*}
     \Vert \bm{X}^o-\bm{X}^{opt}\Vert_F\leq \bar{\alpha}\epsilon.
 \end{equation*}
\end{theorem}
\begin{proof}
    Let $\bm{X}^{opt}=\bm{X}^o+\bm{Z}$, then starting from the block decomposition of $\bm{Z}$
 with respect to $\bm{X}^o$ and the fact $$\Vert \bm{X}^{opt}\Vert_*-\Vert \bm{X}^{opt}\Vert_F\leq \Vert \bm{X}^{o}\Vert_*-\Vert \bm{X}^{o}\Vert_F,$$ we repeat the arguments in the proof of Theorem~\ref{rip_exact} and obtain 
 \begin{equation*}
    \begin{aligned}
&-\sqrt{2} \delta_{\max\{r+\lceil\frac{3}{2}k\rceil,2k\}}(\Vert \bm{Z}_1+\bm{Z}_2+\bm{Z}_{T_1}\Vert_F)\sum_{i\geq 2}^h\Vert \bm{Z}_{T_i}\Vert_F\\
\leq & \langle\mathcal{A}(\bm{Z}_1+\bm{Z}_2+\bm{Z}_{T_1}),\mathcal{A}(\bm{Z})\rangle-\Vert\mathcal{A}(\bm{Z}_1+\bm{Z}_2+\bm{Z}_{T_1})\Vert_2^2
 \end{aligned}  
 \end{equation*}
 and 
  \begin{equation}\label{sum_zi}
     \sum_{i=2}^h\Vert \bm{Z}_{T_i}\Vert_F\leq \frac{\sqrt{2r}+1}{\sqrt{k}-1}\Vert \bm{Z}_1+\bm{Z}_2\Vert_F+\frac{2}{\sqrt{k}-1}\Vert \bm{X}^o_{r_c}\Vert_*.
 \end{equation}
 Easily, we can find 
\begin{equation*}
    \begin{aligned}
      &-\sqrt{2} \delta_{\max\{r+\lceil\frac{3}{2}k\rceil,2k\}}(\Vert \bm{Z}_1+\bm{Z}_2+\bm{Z}_{T_1}\Vert_F)\sum_{i\geq 2}^h\Vert \bm{Z}_{T_i}\Vert_F\nonumber\\
      \leq &2\epsilon \sqrt{1+\delta_{2r+k}}\Vert \bm{Z}_1+\bm{Z}_2+\bm{Z}_{T_1}\Vert_F-(1-\delta_{2r+k})\Vert \bm{Z}_1
      +\bm{Z}_2+\bm{Z}_{T_1}\Vert_F^2,
    \end{aligned}
\end{equation*}
which together with~(\ref{sum_zi}) yields
\begin{equation*}
\begin{aligned}
  & (1-\delta_{2r+k})\Vert \bm{Z}_1+\bm{Z}_2+\bm{Z}_{T_1}\Vert_F\\
   \leq & \sqrt{2} \delta_{\max\{r+\lceil\frac{3}{2}k\rceil,2k\}}\frac{\sqrt{2r}+1}{\sqrt{k}-1}\Vert \bm{Z}_1+\bm{Z}_2\Vert_F+\frac{2\sqrt{2} \delta_{\max\{r+\lceil\frac{3}{2}k\rceil,2k\}}}{\sqrt{k}-1}\Vert \bm{X}^o_{r_c}\Vert_*
   +2\sqrt{1+\delta_{2r+k}}\epsilon.
\end{aligned}
\end{equation*}
Since 
\begin{equation*}
    \beta:= \cfrac{\sqrt{2}\delta_{\max\{r+\lceil\frac{3}{2}k\rceil,2k\}}(\sqrt{2r}+1)}{(1-\delta_{2r+k})(\sqrt{k}-1)}<1
\end{equation*}
and $\Vert \bm{Z}_1+\bm{Z}_2\Vert_F\leq \Vert \bm{Z}_1+\bm{Z}_2+\bm{Z}_{T_1}\Vert_F,$ we have 
\begin{equation}\label{z123_ineq}
\begin{aligned}
 \Vert \bm{Z}_1+\bm{Z}_2+\bm{Z}_{T_1}\Vert_F
 \leq \cfrac{2\sqrt{2} \delta_{\max\{r+\lceil\frac{3}{2}k\rceil,2k\}}}{(1-\delta_{2r+k})(\sqrt{k}-1)(1-\beta)}\Vert \bm{X}^o_{r_c}\Vert_*+\cfrac{2\sqrt{1+\delta_{2r+k}}}{(1-\delta_{2r+k})(1-\beta)}\epsilon.
 \end{aligned}
\end{equation}
Therefore, direct calculation yields
\begin{equation*}
    \begin{aligned}
        \Vert \bm{Z}\Vert_F\leq &\Vert \bm{Z}_1+\bm{Z}_2+\bm{Z}_{T_1}\Vert_F+\sum_{i\geq 2}^h\Vert \bm{Z}_{T_i}\Vert_F\\
        \leq & \frac{\sqrt{k}+\sqrt{2r}}{\sqrt{k}-1}\Vert \bm{Z}_1+\bm{Z}_2+\bm{Z}_{T_1}\Vert_F+\frac{2}{\sqrt{k}-1}\Vert \bm{X}^o_{r_c}\Vert_*\\
        \leq & \alpha\Vert \bm{X}^o_{r_c}\Vert_*+\bar{\alpha}\epsilon.
    \end{aligned}
\end{equation*}
Here the first inequality comes from~(\ref{sum_zi}), and the last inequality comes from~(\ref{z123_ineq}) and the definitions of $\alpha$ and $\bar{\alpha}$. Additionally, from the above discussion, the desired result is easily obtained when $\bm{X}^o$ is $r$-rank matrix. Thus, we complete the proof. 
 \end{proof}
 
\section{Main Results on Unconstrained Minimization} \label{sec:some_T}
\subsection{Relation Between Problem (\ref{exact_pro}) and Problem (\ref{unonstrainpro})}
 We now show that in some sense, problem~(\ref{exact_pro}) can be solved via solving problem~(\ref{unonstrainpro}). We note that the regularization term $\Vert \bm{X}\Vert_*-\Vert \bm{X}\Vert_F$ is nonconvex and nonsmooth, hence the result is nontrivial. 
  \begin{theorem}
Assume linear map $\mathcal{A}$ obeys $\delta_1<1$. Let $\{\lambda_n\}$ be a decreasing sequence of positive numbers with $\lambda_n\rightarrow 0$ and $\bar{\bm{X}}_{\lambda_n}$ be the optimal solution of the problem (\ref{unonstrainpro}) with $\lambda=\lambda_n$. If problem~(\ref{exact_pro}) is feasible, then sequence $\{\bar{\bm{X}}_{\lambda_n}\}$ is bounded and any of its accumulation points is the optimal solution of the problem (\ref{exact_pro}).
    \end{theorem}
\begin{proof}
 Let $\bar{\bm{X}}$ be any feasible point of problem~(\ref{exact_pro}), then $\mathcal{A}(\bar{\bm{X}})=\bm{b}$. Since $\bar{\bm{X}}_{\lambda_n}$ is the optimal solution of problem~(\ref{unonstrainpro}) with respect to $\lambda=\lambda_n$, we have 
    \begin{equation}\label{equpper}
    \begin{aligned}
\Vert\bar{\bm{X}}_{\lambda_n}\Vert_*-\Vert\bar{\bm{X}}_{\lambda_n}\Vert_F+\frac{1}{2\lambda_n}\Vert\mathcal{A}(\bar{\bm{X}}_{\lambda_n})-\bm{b}\Vert_2^2
        \leq \Vert\bar{\bm{X}}\Vert_*-\Vert\bar{\bm{X}}\Vert_F.
           \end{aligned}
    \end{equation}
  Since $\lambda_n\rightarrow 0$ as $n\rightarrow \infty$, the sequences $\{\Vert\bar{\bm{X}}_{\lambda_n}\Vert_*-\Vert\bar{\bm{X}}_{\lambda_n}\Vert_F\}$ and $\{\Vert\mathcal{A}(\bar{\bm{X}}_{\lambda_n})-\bm{b}\Vert_2\}$ converge to zero and so they are bounded.
   
    Next we will show that the sequence $\{\bar{\bm{X}}_{\lambda_n}\}$ is bounded and hence it has at least one accumulation point. 
Denote $\bm{U}_{\lambda_n}\mathcal{D}(\sigma(\bar{\bm{X}}_{\lambda_n}))\bm{V}_{\lambda_n}$ as the SVD of $\bar{\bm{X}}_{\lambda_n}$. Define $$\bar{\bm{X}}_{\lambda_n}^1:=\bm{U}_{\lambda_n}\mathcal{D}(\sigma^1(\bar{\bm{X}}_{\lambda_n}))\bm{V}_{\lambda_n},$$and$$ \bar{\bm{X}}_{\lambda_n}^2:=\bm{U}_{\lambda_n}\mathcal{D}(\sigma^2(\bar{\bm{X}}_{\lambda_n}))\bm{V}_{\lambda_n},$$ where $\sigma^1_i(\bar{\bm{X}}_{\lambda_n})=\sigma_i(\bar{\bm{X}}_{\lambda_n})$ for $i\leq 2$, $\sigma_i^2(\bar{\bm{X}}_{\lambda_n})=\sigma_i(\bar{\bm{X}}_{\lambda_n})$ for $i>2$ and others are $0$. Likewise, let us define matrices $$\bm{Y}^1_{\lambda_n}:=\bm{U}_{\lambda_n}\mathcal{D}(\bm{y}^1_{\lambda_n})\bm{V}_{\lambda_n},~~ \bm{Y}^2_{\lambda_n}:=\bm{U}_{\lambda_n}\mathcal{D}(\bm{y}^2_{\lambda_n})\bm{V}_{\lambda_n},$$ where $(\bm{y}^1_{\lambda_n})_i=\sigma_i(\bar{\bm{X}}_{\lambda_n})$ for $i=1$, $(\bm{y}^2_{\lambda_n})_i=\sigma_i(\bar{\bm{X}}_{\lambda_n})$ for $i>1$ and others are $0$. Clearly, we have
$$\bar{\bm{X}}_{\lambda_n}=\bar{\bm{X}}_{\lambda_n}^1+\bar{\bm{X}}_{\lambda_n}^2=\bm{Y}^1_{\lambda_n}+\bm{Y}^2_{\lambda_n}.$$
By direct computation, it follows from Lemma~\ref{lema_ineq_1-2} (b) that
\begin{equation*}\label{eqlower}
    \begin{aligned}
        \Vert \bar{\bm{X}}_{\lambda_n}\Vert_*-\Vert \bar{\bm{X}}_{\lambda_n}\Vert_F
\geq&\Vert\bar{\bm{X}}_{\lambda_n}^1\Vert_*+\Vert\bar{\bm{X}}_{\lambda_n}^2\Vert_*-\Vert\bar{\bm{X}}_{\lambda_n}^1\Vert_F-\Vert\bar{\bm{X}}_{\lambda_n}^2\Vert_F\\
\geq&\Vert\bar{\bm{X}}_{\lambda_n}^1\Vert_*-\Vert\bar{\bm{X}}_{\lambda_n}^1\Vert_F\\
\geq&(2-\sqrt{2})\sigma_2(\bar{\bm{X}}_{\lambda_n}), 
    \end{aligned}
\end{equation*}
which, together with the boundedness of the sequence $\{\Vert\bar{\bm{X}}_{\lambda_n}\Vert_*-\Vert\bar{\bm{X}}_{\lambda_n}\Vert_F\}$, implies the sequence $\{\Vert\bm{Y}^2_{\lambda_n}\Vert_F\}$ is bounded. Hence, the inequality
\begin{equation*}
    \Vert\mathcal{A}(\bm{Y}^2_{\lambda_n})-\bm{b}\Vert_2\leq\Vert\mathcal{A}\Vert\Vert\bm{Y}^2_{\lambda_n}\Vert_F+\Vert\bm{b}\Vert_2
\end{equation*}
yields the sequence $\{\Vert\mathcal{A}(\bm{Y}^2_{\lambda_n})-\bm{b}\Vert_2\}$ is bounded where $\Vert\mathcal{A}\Vert$ is the operator norm of linear map $
\mathcal{A}$. As a result, since $\{\Vert\mathcal{A}(\bar{\bm{X}}_{\lambda_n})-\bm{b}\Vert_2\}$ is bounded and
   \begin{equation*}
      \Vert\mathcal{A}(\bm{Y}^1_{\lambda_n})\Vert_2\leq  \Vert\mathcal{A}(\bar{\bm{X}}_{\lambda_n})-b\Vert_2+\Vert\mathcal{A}(\bm{Y}^2_{\lambda_n})-\bm{b}\Vert_2,
  \end{equation*} 
 we obtain that the sequence $\{\Vert\mathcal{A}(\bm{Y}^1_{\lambda_n})\Vert_2\}$ is bounded. Furthermore, since linear map $\mathcal{A}$ obeys $\delta_1<1$, it follows from (\ref{ripmap})  that the sequence $\{\Vert\bm{Y}^1_{\lambda_n}\Vert_F\}$ is also bounded. Hence, the boundedness of $\{\bar{\bm{X}}_{\lambda_n}\}$ can be seen from $\bar{\bm{X}}_{\lambda_n}=\bm{Y}^1_{\lambda_n}+\bm{Y}^2_{\lambda_n}$ and the boundedness of $\{\Vert\bm{Y}^1_{\lambda_n}\Vert_F\}$ and $\{\Vert\bm{Y}^2_{\lambda_n}\Vert_F\}$.
  
  Besides, the inequality~(\ref{equpper}) shows
    \begin{equation*}
    \frac{1}{2}\Vert\mathcal{A}(\bar{\bm{X}}_{\lambda_n})-\bm{b}\Vert_2^2\leq \lambda_n(\Vert \bar{\bm{X}}\Vert_*-\Vert\bar{\bm{X}}\Vert_F)
     \end{equation*}
   for any $\lambda_n\rightarrow 0$. Let $\bm{X}^*$ be any accumulation point of the sequence $\{\bar{\bm{X}}_{\lambda_n}\}$. Then we can derive that $\mathcal{A}(\bm{X}^*)=\bm{b}$, which  combines $\Vert \bm{X}^*\Vert_*-\Vert \bm{X}^*\Vert_F\leq \Vert\bar{\bm{X}}\Vert_*-\Vert\bar{\bm{X}}\Vert_F$ and the arbitrariness of $\bar{\bm{X}}$ lead to $\bm{X}^*$ being the optimal solution of problem~(\ref{exact_pro}). This completes the proof.
\end{proof}

\subsection{Recovery Error Estimation for Problem~(\ref{unonstrainpro})}
In this subsection, we provide some theoretical investigations to guarantee the robust recovery of the regularized $L_{1-2}$ minimization problem~(\ref{unonstrainpro}). Let us start with some powerfully technical tools used in the proof of our main results. The following lemma states an elementary geometric fact: any point in a ploytope can be represented as a convex combination of sparse vectors.
\begin{lemma}[Sparse Representation of a Polytope~\cite{cai2013sparse}]\label{sparse_lemma}
  For a positive number $\alpha$ and a positive integer $s$, define the polytope $T(\alpha,s)\subset\mathbb{R}^p$ by 
  \begin{equation*}
   T(\alpha,s)=\{\bm{v}\in\mathbb{R}^p:\Vert \bm{v}\Vert_{\infty}\leq\alpha,\Vert \bm{v}\Vert_1\leq s\alpha\}.   
  \end{equation*}
 For any $\bm{v}\in\mathbb{R}^p$, define the set of sparse vectors $\mathcal{U}(\alpha,s,\bm{v})\subset\mathbb{R}^p$ by 
 \begin{equation}\label{define_U}
 \begin{aligned}
    \mathcal{U}(\alpha,s,\bm{v})=\{\bm{u}\in\mathbb{R}^p: \textup{supp}(\bm{u})\subseteq \textup{supp}(\bm{v}), \Vert \bm{u}\Vert_0\leq s, \Vert \bm{u}\Vert_1=\Vert \bm{v}\Vert_1,\Vert \bm{u}\Vert_{\infty}\leq \alpha\}. 
    \end{aligned}
 \end{equation}
 Then $\bm{v}\in T(\alpha,s)$ if and only if $\bm{v}$ is in the convex hull of $\mathcal{U}(\alpha,s,\bm{v})$. In particularly, any $\bm{v}\in T(\alpha,s)$ can be expressed as 
 \begin{equation*}
     \bm{v}=\sum_{i=1}^{N} \lambda_i \bm{u}_i
 \end{equation*}
 for some positive integer $N$, where $\bm{u}_i\in \mathcal{U}(\alpha,s,\bm{v})$ and 
 \begin{equation}\label{lambdai}
     \sum_{i=1}^{N}\lambda_i=1\quad \text{with}\ \ 0\leq\lambda_i\leq 1.
 \end{equation}
\end{lemma}
The following lemma established in~\cite{cai2013sharp} is also necessary which give an inequality between the sums of the $\alpha$\,th power of two sequences of nonnegative numbers based on the inequality of their sums.
\begin{lemma}\label{befo_lemma}
    Suppose $m\geq r$, $a_1\geq a_2\geq\cdots\geq a_m\geq 0$, and $\sum_{i=1}^{r}a_i\geq\sum_{i=r+1}^m a_i$, then for all $\alpha\geq 1$,
    \begin{equation*}
        \sum_{j=r+1}^m a_j^{\alpha}\leq\sum_{i=1}^{r} a_i^{\alpha}.
    \end{equation*}
    More generally, suppose $a_1\geq a_2\geq\cdots\geq a_m\geq 0$, $\eta\geq 0$ and $\sum_{i=1}^{r}a_i+\eta\geq\sum_{i=r+1}^m a_i$, then for all $\alpha\geq 1$,
    \begin{equation*}
        \sum_{j=r+1}^m a_j^{\alpha}\leq r \left(\sqrt[\alpha]{\cfrac{\sum_{i=1}^r a_i^{\alpha}}{r}}+\frac{\eta}{r}\right)^{\alpha}.
    \end{equation*}
\end{lemma}
With the above preparation, we may state and prove the critical results, which play key role in recovery estimation of regularized minimization. 
\begin{lemma}\label{upper_hT}
    Let $k$ be a positive integer and linear map $\mathcal{A}$ obey the $tk$-order RIP with $\delta_{tk}\in(0,1)$ for certain integer $t>1$. Then for any subset $T\subseteq[n]$ with $\vert T\vert\leq k$ and any matrix $\bm{H}\in\mathbb{R}^{m\times n}$, we have 
    \begin{equation}\label{hkleq}
        \Vert \bm{h}_T\Vert_2\leq \frac{\beta_1}{\sqrt{k}}\Vert \bm{h}_{T^C}\Vert_1+\gamma_1\Vert\mathcal{A}(\bm{H})\Vert_2,
    \end{equation}
    where $\bm{h}\in\mathbb{R}^d$ is the singular value vector of $\bm{H}$ and 
    \begin{equation}\label{para1}
        \beta_1=\frac{\delta_{tk}}{\sqrt{(t-1)[1-\delta_{tk}^2]}},\quad 
        \gamma_1=\frac{2}{(1-\delta_{tk})\sqrt{1+\delta_{tk}}}.
    \end{equation}
\end{lemma}
\begin{proof}
 For given $\bm{H}\in\mathbb{R}^{m\times n}$, let $\bm{U}\mathcal{D}(\bm{h})\bm{V}^{\top}$ be SVD of $\bm{H}$, and for  given $t>1$, define 
\begin{align}
  T_1=\left\{i\in T^c\big|\vert\left( \bm{h}_{T^c}\right)_i\vert>\cfrac{\Vert \bm{h}_{T^c}\Vert_1}{(t-1)k}\right\},\label{T1eq}\\
  T_2=\left\{i\in T^c\big|\vert\left( \bm{h}_{T^c}\right)_i\vert\leq\cfrac{\Vert \bm{h}_{T^c}\Vert_1}{(t-1)k}\right\}\label{T2eq}.
\end{align}
Note that $T^c=T_1\cup T_2$ which yields 
\begin{equation}\label{Eh}  \bm{h}=\bm{h}_{T}+\bm{h}_{T^c}=\bm{h}_{T}+\bm{h}_{T_1}+\bm{h}_{T_2}.
\end{equation}
Besides, the fact $T_1\cap T=\emptyset$ implies $\Vert \bm{h}_T\Vert_2\leq\Vert \bm{h}_{T\cup T_1}\Vert_2$. Thus, to show~(\ref{hkleq}), it suffices to show that 
\begin{equation}\label{ineq_final}
    \Vert \bm{h}_{T\cup T_1}\Vert_2\leq \frac{\beta_1}{\sqrt{k}}\Vert \bm{h}_{T^c}\Vert_1+\gamma_1\Vert\mathcal{A}(\bm{H})\Vert_2,
\end{equation}
where $\beta_1$ and $\gamma_1$ are defined as~(\ref{para1}). We can apply Lemma~\ref{sparse_lemma} to show (\ref{ineq_final}) to be true. 

For this purpose, we first show that 
\begin{equation}\label{numT1}\vert T_1\vert<(t-1)k.\end{equation}
Obviously, the above inequality holds if $T_1=\emptyset$. Otherwise, we can apply~(\ref{T1eq}) to get that
\begin{equation*}
    \Vert \bm{h}_{T_1}\Vert_1>\vert T_1\vert \cfrac{\Vert \bm{h}_{T^c}\Vert_1}{(t-1)k}=\vert T_1\vert \cfrac{\Vert \bm{h}_{T_1\cup T_2}\Vert_1}{(t-1)k}\geq \vert T_1\vert \cfrac{\Vert \bm{h}_{T_1}\Vert_1}{(t-1)k},
\end{equation*}
which implies that (\ref{numT1}) holds. 

We now turn to show (\ref{ineq_final}) holds. The relation $T^c=T_1\cup T_2$ and $T_1\cap T_2=\emptyset$, along with the expression~(\ref{T1eq}), indicates that 
\begin{equation*}
\begin{aligned}
    \Vert \bm{h}_{T_2}\Vert_1&=\Vert \bm{h}_{T^c}\Vert_1-\Vert \bm{h}_{T_1}\Vert_1\\
    &\leq \Vert \bm{h}_{T^c}\Vert_1-\vert T_1\vert \cfrac{\Vert \bm{h}_{T^c}\Vert_1}{(t-1)k}\\
    &=((t-1)k-\vert T_1\vert)\cfrac{\Vert \bm{h}_{T^c}\Vert_1}{(t-1)k}.
    \end{aligned}
\end{equation*}
Since $tk$ is an integer, it follows from (\ref{numT1}) that $(t-1)k-\vert T_1\vert$ is a positive integer. 
Additionally, (\ref{T2eq})  implies that 
\begin{equation*}\Vert \bm{h}_{T_2}\Vert_{\infty}\leq \cfrac{\Vert \bm{h}_{T^c}\Vert_1}{(t-1)k}.
\end{equation*}
By setting 
\begin{equation*}
    \bm{v}=\bm{h}_{T_2},\ \alpha=\cfrac{\Vert \bm{h}_{T^c}\Vert_1}{(t-1)k},\ s=(t-1)k-\vert T_1\vert
\end{equation*}
and applying Lemma~\ref{sparse_lemma}, we have 
\begin{equation}\label{lambdaui}
    \bm{h}_{T_2}=\sum_{i=1}^{N} \lambda_i \bm{u}_i
\end{equation}
for some positive integer $N$, where $\sum_{i=1}^N\lambda_i=1$ with $0\leq\lambda_i\leq 1$ and 
\begin{equation}
     \bm{u}_i\in \mathcal{U}\left(\cfrac{\Vert \bm{h}_{T^c}\Vert_1}{(t-1)k},(t-1)k-\vert T_1\vert,\bm{h}_{T_2}\right)
\end{equation}
with $\mathcal{U}$ being defined as~(\ref{define_U}). Moreover, the expression~(\ref{define_U}) implies that 
\begin{equation}\label{norm_ui}
\begin{aligned}
    \Vert\bm{u}_i\Vert_2^2&\leq \Vert \bm{u}_i\Vert_0\Vert \bm{u}_i\Vert_{\infty}^2\\
    &\leq ((t-1)k-\vert T_1\vert) \left(\cfrac{\Vert \bm{h}_{T^c}\Vert_1}{(t-1)k}\right)^2\\
    &\leq \cfrac{(\Vert \bm{h}_{T^c}\Vert_1)^2}{(t-1)k}.
    \end{aligned}
\end{equation}
For simplicity, denote 
\begin{align}
    \bm{v}_i:=(1+\delta_{tk})\bm{h}_{T\cup T_1}+\delta_{tk}\bm{u}_i,\label{define_vi}\\
    \bar{\bm{v}}_i:=(1-\delta_{tk})\bm{h}_{T\cup T_1}-\delta_{tk}\bm{u}_i.\label{define_barvi}  
\end{align}
The relation $\Vert \bm{u}_i\Vert_0\leq (t-1)k-\vert T_1\vert$ combined with~(\ref{numT1}) and the fact $\vert T\vert\leq k$ yield that $\bm{v}_i$ and $\bar{\bm{v}}_i$ are $tk$-sparse for $1\leq i\leq N$. Thus, we have
\begin{equation}\label{inequalong}
    \begin{aligned}
        &\sum_{i=1}^N\lambda_i\left[\Vert \mathcal{A}\left(\bm{U}\mathcal{D}(\bm{v}_i)\bm{V}^{\top}\right)\Vert_2^2-\Vert \mathcal{A}\left(\bm{U}\mathcal{D}(\bar{\bm{v}}_i)\bm{V}^{\top}\right)\Vert_2^2\right]\\
        =& \sum_{i=1}^N\lambda_i\left[ \Vert \mathcal{A}\left(\bm{U}\mathcal{D}(\bm{h}_{T\cup T_1})\bm{V}^{\top}\right)+\delta_{tk}\mathcal{A}\left(\bm{U}\mathcal{D}(\bm{u}_i+\bm{h}_{T\cup T_1})\bm{V}^{\top}\right)\Vert_2^2-
        \Vert \mathcal{A}\left(\bm{U}\mathcal{D}(\bm{h}_{T\cup T_1})\bm{V}^{\top}\right)\right.\\ 
        &\left.-\delta_{tk}\mathcal{A}\left(\bm{U}\mathcal{D}(\bm{u}_i+\bm{h}_{T\cup T_1})\bm{V}^{\top}\right)\Vert_2^2\right]\\
        =&\sum_{i=1}^N\lambda_i\left[4\delta_{tk}\left( \mathcal{A}(\bm{U}\mathcal{D}(\bm{h}_{T\cup T_1})\bm{V}^{\top})\right)^{\top}\mathcal{A}\left(\bm{U}\mathcal{D}(\bm{u}_i+\bm{h}_{T\cup T_1})\bm{V}^{\top}\right)
        \right]\\
        \overset{(a)}{=}& 4\delta_{tk}\left( \mathcal{A}(\bm{U}\mathcal{D}(\bm{h}_{T\cup T_1})\bm{V}^{\top})\right)^{\top}\mathcal{A}\left(\bm{U}\mathcal{D}(\sum_{i=1}^{N}\lambda_i\bm{u}_i+\bm{h}_{T\cup T_1})\bm{V}^{\top}\right)\\
        \overset{(b)}{=}& 4\delta_{tk}\left( \mathcal{A}(\bm{U}\mathcal{D}(\bm{h}_{T\cup T_1})\bm{V}^{\top})\right)^{\top}\mathcal{A}\left(\bm{U}\mathcal{D}(\bm{h})\bm{V}^{\top}\right)\\
        =& 4\delta_{tk}\left( \mathcal{A}(\bm{U}\mathcal{D}(\bm{h}_{T\cup T_1})\bm{V}^{\top})\right)^{\top}\mathcal{A}(\bm{H})\\
        \overset{(c)}{\leq}&4\delta_{tk}\sqrt{1+\delta_{tk}}\Vert \bm{h}_{T\cup T_1}\Vert_2\Vert \mathcal{A}(\bm{H})\Vert_2,
    \end{aligned}
\end{equation}
where (a) follows from~(\ref{lambdai}), (b) follows from~(\ref{Eh}) and~(\ref{lambdaui}), (c) follows from Cauchy-Schwartz inequality and~(\ref{ripmap}). 

Next we will give a lower bound of the first formula in~(\ref{inequalong}). That is 
\begin{equation*}
    \begin{aligned}
      &\sum_{i=1}^N\lambda_i\left[\Vert \mathcal{A}\left(\bm{U}\mathcal{D}(\bm{v}_i)\bm{V}^{\top}\right)\Vert_2^2-\Vert \mathcal{A}\left(\bm{U}\mathcal{D}(\bar{\bm{v}}_i)\bm{V}^{\top}\right)\Vert_2^2\right]\\
      \overset{(a)}{\geq} & \sum_{i=1}^N\lambda_i(1-\delta_{tk})\Vert \bm{v}_i\Vert_2^2-\sum_{i=1}^N \lambda_i(1+\delta_{tk})\Vert \bar{\bm{v}}_i\Vert_2^2\\
      \overset{(b)}{=}& \sum_{i=1}^N\lambda_i\left(2(\delta_{tk}-\delta_{tk}^3)\Vert \bm{h}_{T\cup T_1}\Vert_2^2-2\delta_{tk}^3\Vert \bm{u}_i\Vert_2^2\right)\\
      \overset{(c)}{\geq}&2(\delta_{tk}-\delta_{tk}^3)\Vert \bm{h}_{T\cup T_1}\Vert_2^2-2\delta_{tk}^3 \cfrac{(\Vert \bm{h}_{T^c}\Vert_1)^2}{(t-1)k},
\end{aligned}
\end{equation*}
where (a) follows from~(\ref{ripmap}), (b) follows from~(\ref{define_vi}) and~(\ref{define_barvi}), (c) follows from~(\ref{lambdai}) and~(\ref{norm_ui}). The inclusion together with (\ref{inequalong}) indicates that
\begin{equation*}
\begin{aligned}
    2(\delta_{tk}-\delta_{tk}^3)\Vert \bm{h}_{T\cup T_1}\Vert_2^2-2\delta_{tk}^3 \cfrac{(\Vert \bm{h}_{T^c}\Vert_1)^2}{(t-1)k}
    \leq 4\delta_{tk}\sqrt{1+\delta_{tk}}\Vert \bm{h}_{T\cup T_1}\Vert_2\Vert \mathcal{A}(\bm{H})\Vert_2,
    \end{aligned}
\end{equation*}
which is exactly 
\begin{equation*}
    (1-\delta_{tk}^2)\Vert \bm{h}_{T\cup T_1}\Vert_2^2-2\sqrt{1+\delta_{tk}}\Vert \bm{h}_{T\cup T_1}\Vert_2\Vert \mathcal{A}(\bm{H})\Vert_2
    -\delta_{tk}^2 \cfrac{(\Vert \bm{h}_{T^c}\Vert_1)^2}{(t-1)k}\leq 0.
\end{equation*}
Hence we have
\begin{equation*}
    \begin{aligned}
      \Vert \bm{h}_{T\cup T_1}\Vert_2
 \leq& \bigg[\left(  (2\sqrt{1+\delta_{tk}}\Vert \mathcal{A}(\bm{H})\Vert_2)^2
 +4\delta_{tk}^2(1-\delta_{tk}^2)          \cfrac{(\Vert \bm{h}_{T^c}\Vert_1)^2}{(t-1)k}  \right)^{\frac{1}{2}}+2\sqrt{1+\delta_{tk}}\Vert \mathcal{A}(\bm{H})\Vert_2 \bigg]\bigg/2(1-\delta_{tk}^2)\\
 \leq& \cfrac{\delta_{tk}}{\sqrt{(1-\delta_{tk}^2)(t-1)k}}\Vert \bm{h}_{T^c}\Vert_1+\cfrac{2}{(1-\delta_{tk})\sqrt{1+\delta_{tk}}}\Vert \mathcal{A}(\bm{H})\Vert_2,  
    \end{aligned}
\end{equation*}
where the last inequality follows from $\sqrt{x^2+y^2}\leq x+y$ for any $x\geq 0$ and $y\geq 0$.

By virtue of~(\ref{para1}) and~(\ref{ineq_final}), we obtain the desired results. This completes the proof.
\end{proof}
\begin{lemma}\label{upper_hc}
Let $k$ be a positive integer and the desired matrix $\bm{X}^o$ satisfy $\mathcal{A}(\bm{X}^o)+\bm{s}=\bm{b}$ with perturbation $\Vert\bm{s}\Vert_2\leq\epsilon$. Set $\bm{H}:=\bm{X}^*-\bm{X}^o$ with $\bm{X}^*\in\arg\min\mathcal{J}(\bm{X})$. Denote $\bm{U}_o\mathcal{D}(\bm{x}^0)\bm{V}_o^{\top}$ and $\bm{U}_H\mathcal{D}(\bm{h})\bm{V}_H^{\top}$ as the SVD of $\bm{X}^o$ and $\bm{H}$, respectively. Define $\Omega:=\textup{supp}(\bm{x}^o_{\max(k)})$ and $\Gamma:=\textup{supp}(\bm{h}_{\max(k)})$, then we have 
\begin{equation}\label{inequ_h}
\begin{aligned}
    \Vert \mathcal{A}(\bm{H})\Vert_2^2-2\epsilon \Vert\mathcal{A}(\bm{H})\Vert_2
    \leq 2\lambda(2\Vert \bm{X}^o\Vert_*+\Vert\bm{h}_{\Gamma}\Vert_1-\Vert\bm{h}_{\Gamma^c}\Vert_1+\Vert\bm{h}\Vert_2)
\end{aligned}
\end{equation}
and
\begin{equation}\label{inequ_hc}
   \Vert\bm{h}_{\Gamma^c}\Vert_1\leq  2\Vert \bm{X}^o\Vert_*+\Vert\bm{h}_{\Gamma}\Vert_1+\Vert\bm{h}\Vert_2+\cfrac{\epsilon}{\lambda}\Vert\mathcal{A}(\bm{H})\Vert_2.
\end{equation}
\end{lemma}
\begin{proof}
According to the definition of $\bm{X}^*$, we obtain that $\mathcal{J}(\bm{X}^*)\leq \mathcal{J}(\bm{X}^o)$, i.e., 
\begin{equation*}
\begin{aligned}
    \Vert \bm{X}^*\Vert_*-\Vert \bm{X}^*\Vert_F+\cfrac{1}{2\lambda}\Vert \mathcal{A}(\bm{X}^*)-\bm{b}\Vert_2^2
    \leq \Vert \bm{X}^o\Vert_*-\Vert \bm{X}^o\Vert_F+\cfrac{1}{2\lambda}\Vert \mathcal{A}(\bm{X}^o)-\bm{b}\Vert_2^2,
    \end{aligned}
\end{equation*}
which after simplification gives 
\begin{equation}\label{inequa_objec}
\begin{aligned}
    \cfrac{1}{2\lambda}\left(\Vert \mathcal{A}(\bm{H})\Vert_2^2-2\langle \bm{s},\mathcal{A}(\bm{H})\rangle\right)
    \leq \Vert \bm{X}^o\Vert_*-\Vert \bm{X}^o\Vert_F-\Vert \bm{X}^*\Vert_*+\Vert \bm{X}^*\Vert_F.
    \end{aligned}
\end{equation}
Denote the left and right sides of~(\ref{inequa_objec}) as $\rho_l$ and $\rho_r$ respectively. It is known from the Cauchy-Schwartz inequality that 
\begin{equation}\label{ineuqrho_l}
\begin{aligned}
    \rho_l\geq&\frac{1}{2\lambda}\left(\Vert \mathcal{A}(\bm{H})\Vert_2^2-2\Vert\bm{s}\Vert_2\Vert\mathcal{A}(\bm{H})\Vert_2\right)\\
    \geq& \frac{1}{2\lambda}\left(\Vert \mathcal{A}(\bm{H})\Vert_2^2-2\epsilon\Vert\mathcal{A}(\bm{H})\Vert_2\right).
    \end{aligned}
\end{equation}
Besides, direct calculations lead to the expression of $\bm{H}+\bm{X}^o=\bm{A}+\bm{B}$ with
\begin{equation*}
    \bm{A}=\bm{U}_o\mathcal{D}(\bm{x}^o_{\Omega})\bm{V}_o^{\top}+\bm{U}_H\mathcal{D}(\bm{h}_{\Omega})\bm{V}_H^{\top}
\end{equation*}
and
\begin{equation*}
    \bm{B}=\bm{U}_o\mathcal{D}(\bm{x}^o_{\Omega^c})\bm{V}_o^{\top}+\bm{U}_H\mathcal{D}(\bm{h}_{\Omega^c})\bm{V}_H^{\top}.
\end{equation*}
Then, we have
\begin{align}\label{inequrho_r}
    \rho_r
    =&\Vert \bm{X}^o\Vert_*-\Vert \bm{X}^o\Vert_F-\Vert \bm{H}+\bm{X}^o\Vert_*+\Vert \bm{H}+\bm{X}^o\Vert_F\notag\\
    \leq& \Vert \bm{X}^o\Vert_*-\Vert \bm{X}^o\Vert_F-\Vert \bm{B}\Vert_*+\Vert \bm{A}\Vert_*+\Vert \bm{H}\Vert_F+\Vert \bm{X}^o\Vert_F\notag\\
    \leq& \Vert \bm{X}^o\Vert_*-\Vert \bm{h}_{\Omega^c}\Vert_1+\Vert \bm{x}^o_{\Omega^c}\Vert_1+\Vert \bm{x}^o_{\Omega}\Vert_1+\Vert \bm{h}_{\Omega}\Vert_1+\Vert \bm{h}\Vert_2\notag\\
    \leq& 2\Vert \bm{X}^o\Vert_*+\Vert \bm{h}_{\Gamma}\Vert_1-\Vert \bm{h}_{\Gamma^c}\Vert_1+\Vert \bm{h}\Vert_2.
\end{align}
Hence combining~(\ref{ineuqrho_l}) with~(\ref{inequrho_r}), we can get the desired inequalities (\ref{inequ_h}) and (\ref{inequ_hc}) hold trivially from~(\ref{inequ_h}). 
\end{proof}
We shall focus on investigating the recovery performance of problem~(\ref{unonstrainpro}) and characterizing the recovery errors of this method. 
\begin{theorem}
    Assume that $k$ is a positive integer with $k\geq 6$, and $\bm{b}$ is given by~(\ref{noise_AX}) with $\Vert \bm{s}\Vert_2 \leq \epsilon$. If linear map $\mathcal{A}$ obeys 
    \begin{equation}\label{upper_delta}
        \delta_{tk}<\sqrt{\cfrac{t-1}{t+\theta_k^2-1}}
    \end{equation}
    for a certain integer $t>1$, where $\theta_k$ is given by 
    \begin{equation}\label{thetak}
        \theta_k=\cfrac{\sqrt{k}+\sqrt{2}-1}{\sqrt{k}-\sqrt{2}-1}.
    \end{equation}
 Then we have 
 \begin{equation}
     \Vert \bm{X}^*-\bm{X}^o\Vert_F\leq C_1\Vert \bm{X}^o\Vert_*+C_2\lambda
 \end{equation}
 where $C_1$ and $C_2$ are determined by 
 \begin{equation*}
     C_1=\cfrac{2(\hat{\beta}_1\hat{\gamma}_1+\xi_1)}{\sqrt{k}(1-\beta_1)\hat{\gamma}_1\kappa_1},\quad 
     C_2=\cfrac{2\xi_1\hat{\gamma}_1}{\sqrt{k}(1-\beta_1)\kappa_1},
 \end{equation*}
 with $\hat{\beta}_1$ and $\hat{\gamma}_1$ being given by~(\ref{def_hatbeta1}) and~(\ref{define_hatgamma}), and $\kappa_1$ and $\xi_1$ being given in~(\ref{def_xi1}) and~(\ref{def_kappa1}), respectively.
 \end{theorem}
\begin{proof}
Denote $\bm{H}=\bm{X}^*-\bm{X}^o$. Let $\bm{X}^o=\bm{U}_o\mathcal{D}(\bm{x}^0)\bm{V}_o^{\top}$ and $\bm{H}=\bm{U}_H\mathcal{D}(\bm{h})\bm{V}_H^{\top}$ be the SVD of $\bm{X}^o$ and $\bm{H}$, respectively. Define $\Omega:=\textup{supp}(\bm{x}^o_{\max(k)})$, $\Gamma:=\textup{supp}(\bm{h}_{\max(k)})$ and  $\omega:=2\Vert \bm{X}^o\Vert_*+\Vert\bm{h}\Vert_2$. On the one hand, noting the fact $\Vert\bm{h}_{\Gamma}\Vert_1\leq \sqrt{k}\Vert \bm{h}_{\Gamma}\Vert_2$, together with Lemma~\ref{upper_hT} and Lemma~\ref{upper_hc}, we obtain 
\begin{equation}
    \Vert\bm{h}_{\Gamma}\Vert_2\leq\cfrac{\beta_1}{\sqrt{k}}\left(\omega+\sqrt{k}\Vert\bm{h}_{\Gamma}\Vert_2+\frac{\epsilon}{\lambda}\Vert\mathcal{A}(\bm{H})\Vert_2\right)+\gamma_1\Vert\mathcal{A}(\bm{H})\Vert_2,
\end{equation}
which is exactly 
\begin{equation}
    \Vert\bm{h}_{\Gamma}\Vert_2\leq\cfrac{\beta_1\epsilon+\gamma_1\sqrt{k}\lambda}{(1-\beta_1)\sqrt{k}\lambda}\Vert\mathcal{A}(\bm{H})\Vert_2+\frac{\beta_1}{(1-\beta_1)\sqrt{k}} \omega
\end{equation}
due to $\beta_1<1$ by~(\ref{upper_delta}) and~(\ref{thetak}).
On the other hand, invoking Lemma~\ref{befo_lemma} and~(\ref{inequ_hc}), it yields 
\begin{equation*}
    \Vert \bm{h}_ {\Gamma^c}\Vert_2\leq \sqrt{k}\left(  \sqrt{\cfrac{\Vert\bm{h}_{\Gamma}\Vert_2^2}{k}}+\cfrac{\omega+\frac{\epsilon}{\lambda}\Vert \mathcal{A}(\bm{H})\Vert_2}{k}\right),
\end{equation*}
that is, 
\begin{equation*}
    \Vert \bm{h}_ {\Gamma^c}\Vert_2\leq \Vert\bm{h}_{\Gamma}\Vert_2+\frac{1}{\sqrt{k}}\left(  \omega+\frac{\epsilon}{\lambda}\Vert \mathcal{A}(\bm{H})\Vert_2\right).
\end{equation*}
Define
\begin{equation}\label{def_hatbeta1}
    \hat{\beta}_1:=(\sqrt{2}-1)\beta_1+1.
\end{equation}
Therefore,
we obtain that 
\begin{equation}\label{vert_h_leq}
    \begin{aligned}
\Vert \bm{h}\Vert_2
=&\sqrt{\Vert \bm{h}_ {\Gamma}\Vert_2^2+\Vert \bm{h}_ {\Gamma^c}\Vert_2^2} \\
 \leq& \sqrt{ \Vert \bm{h}_ {\Gamma}\Vert_2^2+ \left(\Vert\bm{h}_{\Gamma}\Vert_2+\frac{1}{\sqrt{k}}(  \omega+\frac{\epsilon}{\lambda}\Vert \mathcal{A}(\bm{H})\Vert_2)\right)^2    }\\
 \leq& \sqrt{2}\Vert \bm{h}_ {\Gamma}\Vert_2+\frac{1}{\sqrt{k}}(\omega+\frac{\epsilon}{\lambda}\Vert \mathcal{A}(\bm{H})\Vert_2)\\
 \leq &\cfrac{\sqrt{2k}\gamma_1\lambda+\epsilon\left((\sqrt{2}-1)\beta_1+1\right)}{\lambda\sqrt{k}(1-\beta_1)}\Vert \mathcal{A}(\bm{H})\Vert_2+\frac{(\sqrt{2}-1)\beta_1+1}{\sqrt{k}(1-\beta_1)}\omega\\
 =&
\cfrac{\sqrt{2k}\gamma_1\lambda+\hat{\beta}_1\epsilon}{\lambda\sqrt{k}(1-\beta_1)}\Vert \mathcal{A}(\bm{H})\Vert_2+\frac{\hat{\beta}_1}{\sqrt{k}(1-\beta_1)} \omega.
    \end{aligned}
\end{equation}
By employing~(\ref{inequ_h}) and Lemma~\ref{upper_hT}, we have  
\begin{equation*}
\begin{aligned}
    \Vert\mathcal{A}(\bm{H})\Vert_2^2-2\epsilon\Vert \mathcal{A}(\bm{H})\Vert_2
    \leq 2 \sqrt{k}\lambda(\frac{\beta_1}{\sqrt{k}}\Vert\bm{h}_{\Gamma^c}\Vert_1+\gamma_1\Vert \mathcal{A}(\bm{H})\Vert_2)-2\lambda\Vert \bm{h}_{\Gamma^c}\Vert_1+2\lambda\omega,
\end{aligned}
\end{equation*}
that is, 
\begin{equation*}
     \Vert\mathcal{A}(\bm{H})\Vert_2^2-(2\sqrt{k}\lambda\gamma_1+2\epsilon)\Vert \mathcal{A}(\bm{H})\Vert_2-2\lambda\omega\leq 2\lambda(\beta_1-1)\Vert\bm{h}_{\Gamma^c}\Vert_1.
\end{equation*}
Define
\begin{equation}\label{define_hatgamma}
    \eta:=\frac{\epsilon}{\lambda},\ \hat{\gamma}_1:=\sqrt{k}\gamma_1+\eta.
\end{equation}
It yields 
\begin{equation*}
    \Vert\mathcal{A}(\bm{H})\Vert_2^2-2\hat{\gamma}_1\lambda\Vert \mathcal{A}(\bm{H})\Vert_2-2\lambda\omega\leq0.
\end{equation*}
Thus, 
\begin{equation*}
    \begin{aligned}
        \Vert\mathcal{A}(\bm{H})\Vert_2&\leq\frac{2\hat{\gamma}_1\lambda+\sqrt{4\hat{\gamma}_1^2\lambda^2+8\lambda\omega}}{2}\\
        &\leq \hat{\gamma}_1\lambda+\sqrt{(\hat{\gamma}_1\lambda+\frac{\omega}{\hat{\gamma}_1})^2}\\
        &=2\hat{\gamma}_1\lambda+\frac{\omega}{\hat{\gamma}_1},
    \end{aligned}
\end{equation*}
which indicates that 
\begin{equation*}
\begin{aligned}
\Vert\bm{h}\Vert_2
\leq&\cfrac{\sqrt{2k}\gamma_1\lambda+\hat{\beta}_1\epsilon}{\lambda\sqrt{k}(1-\beta_1)}(2\hat{\gamma}_1\lambda+\frac{\omega}{\hat{\gamma}_1})+\frac{\hat{\beta}_1}{\sqrt{k}(1-\beta_1)} \omega\\
=&\left[\frac{\hat{\beta}_1}{\sqrt{k}(1-\beta_1)}+\frac{\sqrt{2k}\gamma_1+\hat{\beta}_1\cfrac{\epsilon}{\lambda}}{\sqrt{k}\hat{\gamma}_1(1-\beta_1)}\right] \omega+\frac{\sqrt{2k}\gamma_1+\hat{\beta}_1\cfrac{\epsilon}{\lambda}}{\sqrt{k}(1-\beta_1)}2\hat{\gamma}_1\lambda\\
=& \frac{\hat{\beta}_1(\hat{\gamma}_1+\eta)+\sqrt{2k}\gamma_1}{\sqrt{k}\hat{\gamma}_1(1-\beta_1)}\omega+\frac{2\hat{\gamma}_1(\sqrt{2k}\gamma_1+\hat{\beta}_1\eta)}{\sqrt{k}(1-\beta_1)}\lambda. 
\end{aligned}
\end{equation*}
Define 
\begin{equation}\label{def_xi1}
    \xi_1:=\sqrt{2k}\gamma_1+\hat{\beta}_1\eta.
\end{equation}
Then, we have
\begin{equation*}
\begin{aligned}
    \left(1-\frac{\hat{\beta}_1\hat{\gamma}_1+\xi_1}{\sqrt{k}(1-\beta_1)\hat{\gamma}_1}\right)\Vert\bm{h}\Vert_2   \leq\frac{2(\hat{\beta}_1\hat{\gamma}_1+\xi_1)}{\sqrt{k}(1-\beta_1)\hat{\gamma}_1}\Vert \bm{X}^o\Vert_*+\frac{2\hat{\gamma}_1\xi_1}{\sqrt{k}(1-\beta_1)}\lambda.
\end{aligned}
\end{equation*}
Denote 
\begin{equation}\label{def_kappa1}
    \kappa_1:=1-\frac{\hat{\beta}_1\hat{\gamma}_1+\xi_1}{\sqrt{k}(1-\beta_1)\hat{\gamma}_1},    
\end{equation} 
it is easy to verify that $\kappa_1>0$ by applying the relation~(\ref{upper_delta}) and~(\ref{thetak}) and the fact $\xi_1\leq \sqrt{2}\hat{\gamma}_1$. Therefore,
\begin{equation*}
    \Vert \bm{H}\Vert_F
    \leq\frac{2(\hat{\beta}_1\hat{\gamma}_1+\xi_1)}{\sqrt{k}(1-\beta_1)\hat{\gamma}_1\kappa_1}\Vert \bm{X}^o\Vert_*+\frac{2\hat{\gamma}_1\xi_1}{\sqrt{k}(1-\beta_1)\kappa_1}\lambda
\end{equation*}
and we complete the proof.
\end{proof}

\section{Conclusion}
In this paper, we focused on analysis of theoretical guarantees for reconstruction of low rank matrix from noisy measurements via $L_{*-F}$ minimization. Firstly, we briefly presented the relationship of the optimal solution among several $L_{*-F}$ minimization problems. Secondly, we gave the sufficient conditions of stable recovery and the recovery error estimation for the general constrained $L_{*-F}$ minimization problems. Besides, we considered the robust low rank matrix recovery by regularized $L_{*-F}$ minimization model, approximate error estimation of which was obtained under the framework of powerful RIP tools. To our knowledge, this theoretical result has not been studied before. A further issue worth considering is developing a tighter recovery error estimation for this regularized model. It is also interesting to see if the techniques in this paper can be applied in other settings.


%


\ifCLASSOPTIONcaptionsoff
  \newpage
\fi

\end{document}